\documentclass[11pt,a4paper]{article}
\usepackage{amsmath,amssymb,amsthm,amscd,mathrsfs}
\usepackage{indentfirst}
\usepackage[top=25mm, bottom=30mm, left=23mm, right=23mm]{geometry}

\newtheorem{Def}{\bf Definition}[section]
\newtheorem{Thm}[Def]{\bf Theorem}
\newtheorem{Lem}[Def]{\bf Lemma}

\newtheorem{Pro}[Def]{\bf Proposition}
\newtheorem{Rem}[Def]{\bf Remark}

\newtheorem*{claim}{\bf Claim}

\newtheorem{ThmA}{\bf Theorem}
\newtheorem{CorA}[ThmA]{\bf Corollary}

\newtheorem{ProA}[ThmA]{\bf Proposition}

\newcommand{\R}{\mathbb{R}}
\newcommand{\C}{\mathbb{C}}
\newcommand{\Z}{\mathbb{Z}}

\newcommand{\N}{\mathbb{N}}
\newcommand{\B}{\mathbb{B}}

\newcommand{\M}{\mathbb{M}}

\newcommand{\Ad}{\operatorname{Ad}}
\newcommand{\id}{\text{\rm id}}
\newcommand{\Aut}{\operatorname{Aut}}
\newcommand{\Tr}{\mathord{\text{\rm Tr}}}
\newcommand{\ovt}{\mathbin{\overline{\otimes}}}

\title{\bf Unitary conjugacy for type III subfactors and W$^*$-superrigidity}
\author{Yusuke Isono\thanks{Research Institute for Mathematical Sciences, Kyoto University, 606-8502, Kyoto, Japan \protect \\  E-mail: \texttt{isono@kurims.kyoto-u.ac.jp}}}
\date{}
\begin{document}
\maketitle

\begin{abstract}
	Let $A,B\subset M$ be inclusions of $\sigma$-finite von Neumann algebras such that $A$ and $B$ are images of faithful normal conditional expectations. In this article, we investigate Popa's intertwining condition $A\preceq_MB$ using their modular actions. In the main theorem, we prove that if $A\preceq_MB$ holds, then an intertwining element for $A\preceq_MB$ also intertwines some modular flows of $A$ and $B$. 
As a result, we deduce a new characterization of $A\preceq_MB$ in terms of their continuous cores. 
Using this new characterization, we prove the first W$^*$-superrigidity type result for group actions on amenable factors. 
As another application, we characterize stable strong solidity for free product factors in terms of their free product components.
\end{abstract}

\section{Introduction}\label{Introduction}

	In \cite{Po01}, Sorin Popa obtained the first uniqueness result for certain Cartan subalgebras in non-amenable type II$_1$ factors up to unitary conjugacy. He used this result to compute some invariants of von Neumann algebras and succeeded to give the first examples of type II$_1$ factors which have trivial fundamental groups, solving a long standing open problem in von Neumann algebra theory. 
This breakthrough work led to great progress in the classification of non-amenable von Neumann algebras over the last years, which is now called Popa's \textit{deformation/rigidity theory} (see the surveys \cite{Po06b,Va10,Io17}).

	An important technical 
ingredient in his theory is the \textit{intertwining-by-bimodules} technique \cite{Po01,Po03}. 
Let $M$ be a finite von Neumann algebra and $A,B \subset M$ von Neumann subalgebras. The \textit{intertwining condition}, which will be written as $A \preceq_M B$, is defined as a weaker notion of unitary conjugacy from $A$ into $B$ (see Definition \ref{def corner intertwining}). Popa proved that this condition is equivalent to an analytic condition: non-existence of a net of unitaries in $A$ with a certain convergence condition. 
This equivalence provides a very powerful tool to obtain a unitary conjugacy between certain subalgebras, and it is now regarded as a fundamental tool to study relations between \textit{general} subalgebras in a von Neumann algebra.

	The proof of this analytic characterization relies on the bimodule structure via GNS representations of \textit{traces}. The finiteness assumption of $M$ is hence crucial in this context. However since there are many natural questions for non-tracial von Neumann algebras (more specifically, for type III factors) which should be studied in deformation/rigidity theory, there have been many attempts to generalize the intertwining machinery to type III von Neumann algebras. 
In a joint work with C. Houdayer \cite{HI15}, we succeeded to prove the aforementioned analytic characterization in the case when $A$ is finite (and $B \subset M$ can be general), but the general case is still open. See also \cite{CH08,HR10,HV12,Ue12,Is14,Ue16,BH16} for other partial generalizations of this technique.

	In the present article, we focus on this problem. We will investigate Popa's intertwining condition $A\preceq_MB$ for general inclusions of von Neumann algebras. 
Before proceeding, we prepare some terminology. 
For a (possibly non-unital) inclusion of von Neumann algebras $A\subset M$, we say that $A\subset M$ is \textit{with expectation} if there is a faithful normal conditional expectation $E_A \colon 1_AM1_A \to A$, where $1_A$ is the unit of $A$. 
For any such expectation $E_A$, we say that a faithful normal positive functional $\varphi\in M_*$ is \textit{preserved by $E_A$} if it satisfies $\varphi = \varphi(1_A \cdot 1_A)+ \varphi(1_A^\perp \cdot 1_A^\perp)$ and $\varphi\circ E_A=\varphi$ on $1_AM1_A$, where $1_A^\perp:=1_M-1_A$.

	Now we introduce the main theorem in this article. The theorem shows that the intertwining condition $A\preceq_MB$ is equivalent to the same condition but together with additional conditions on \textit{Tomita--Takesaki's modular actions}. More precisely, an intertwining element, which manages a weak unitary conjugacy for $A\preceq_MB$, also intertwines some modular flows for $A$ and $B$. 
As a result, the condition $A\preceq_MB$ is equivalent to a condition on their continuous cores (see item (3) below). 
This provides new perspective for the intertwining machinery in type III von Neumann algebra theory. 
In the theorem below, $\sigma^\varphi$ is the modular action and $C_\varphi(M)$ is the continuous core of $M$ (with respect to $\varphi\in M_*^+$), see Section \ref{Preliminaries}. Recall that a factor $N$ is a \textit{type $\rm III_1$ factor} if its continuous core is a factor. 
See Definition \ref{def corner intertwining with modular actions} and \ref{def corner intertwining with expectation} for intertwining conditions with modular actions and with conditional expectations.

\begin{ThmA}\label{thmA}
	Let $M$ be $\sigma$-finite von Neumann algebra and $A,B\subset M$ (possibly non-unital) von Neumann subalgebras with expectations. We fix any faithful normal conditional expectation $E_B\colon 1_BM1_B \to B$, any faithful state $\varphi\in M_*$ which is preserved by $E_B$. 
Then the following two conditions are equivalent.
\begin{itemize}
	\item We have $A\preceq_MB$.
	\item We have $(A,\sigma^\psi)\preceq_M(B,\sigma^\varphi)$ for some faithful state $\psi\in M_*$ such that $\sigma_t^\psi(A)=A$ for all $t\in \R$ (or equivalently, such that $\psi$ is preserved by some conditional expectation onto $A$). 
\end{itemize}
Moreover, for any fixed faithful normal conditional expectation $E_{A} \colon 1_AM1_A \to A$, any faithful state $\psi\in M_*$ which is preserved by $E_A$, and any $\sigma$-finite type $\rm III_1$ factor $N$ equipped with a faithful state $\omega\in N_*$, the following conditions are equivalent.
\begin{itemize}
	\item[$\rm (1)$] We have $(A,\sigma^\psi)\preceq_M(B,\sigma^\varphi)$.
	\item[$\rm (2)$] We have $(A,E_A)\preceq_M(B,E_B)$.
	\item[$\rm (3)$] We have $\Pi(C_{\psi\otimes\omega}(A\ovt N)) \preceq_{C_{\varphi\otimes\omega}(M\ovt N)}C_{\varphi\otimes \omega}(B\ovt N)$, where $\Pi\colon C_{\psi\otimes\omega}(M\ovt N) \to C_{\varphi\otimes\omega}(M\ovt N)$ is the canonical $\ast$-isomorphism given by the Connes cocycle.
\end{itemize}
\end{ThmA}
 	The following immediate corollary gives a new characterization of $A\preceq_MB$ in terms of their continuous cores. Since all continuous cores are semifinite, up to cutting down by a finite projection, one can use the analytic characterization of the intertwining condition at the level of continuous cores. 
\begin{CorA}
	Keep the setting as in Theorem \ref{thmA} and fix a type $\rm III_1$ factor $N$ and a faithful state $\omega\in N_*$. Then $A\preceq_MB$ holds if and only if item (3) in Theorem \ref{thmA} holds for some $E_A$ and $\psi$.
\end{CorA}
	We emphasize that this corollary \textit{fails} if we do not take tensor products with a type III$_1$ factor. In fact, there is an inclusion $B\subset M=A$ such that $M\not\preceq _MB$ but $C_\varphi(M) \preceq_{C_\varphi(M)}C_\varphi(B)$ (see \cite[Theorem 4.9]{HI17}). 
Hence the type III$_1$ factor $N$ is necessary.

	Here we explain the idea behind Theorem \ref{thmA}. 
	In \cite{Po04,Po05a}, Popa proved his celebrated cocycle superrigidity theorem. He developed a way of using his intertwining machinery to study cocycles of actions. 
If two discrete group actions $\Gamma \curvearrowright^\alpha M$ and $\Gamma \curvearrowright^\beta M$ on a finite von Neumann algebra $M$ are cocycle conjugate (so that $M\rtimes_\beta \Gamma = M\rtimes_\alpha \Gamma$), then the intertwining condition $\C 1_M \rtimes_{\beta}\Gamma \preceq_{M\rtimes_\alpha\Gamma} \C1_M \rtimes_{\alpha} \Gamma$ is equivalent to a weak conjugacy condition for $\alpha$ and $\beta$ (see Definition \ref{def unital intertwining with actions}).
	In \cite{HSV16}, by assuming the subalgebra $A$ is trivial (but $B\subset M$ can be general), Houdayer, Shlyakhtenko, and Vaes applied this idea to the case of modular actions. They combined it with \textit{Connes cocycles} and deduced a new characterization of intertwining conditions, in terms of their states. 
This new characterization enabled them to identify specific states on von Neumann algebras, and they applied it to the classification of free Araki--Woods factors.

	Our Theorem \ref{thmA} is strongly motivated by these works. In fact, when the subalgebra $A$ is finite, Theorem \ref{thmA} can be proved (without tensoring a type III$_1$ factor) by developing ideas in these works. 
Hence the main interest of Theorem \ref{thmA} is the case that $A$ is of type III. 
It is technically more challenging, since both proofs of \cite{Po04,Po05a} and \cite{HSV16} are no longer adapted. 
We will use another characterization of $A\preceq_MB$ which holds without the finiteness assumption (see Theorem \ref{thm corner intertwining}(2)). 
By taking tensor products with a type III$_1$ factor $N$ and by analyzing operator valued weights on basic constructions, we will connect this condition on $M$ to the one of $C_\varphi(M \ovt N)$. See Lemma \ref{III1 factor tensor lemma} and \ref{key lemma} for the use of type III$_1$ factors.

\subsection*{Application: W$^*$-superrigidity for actions on amenable factors}

	Our first application of Theorem \ref{thmA} is on \textit{W$^*$-superrigidity} of group actions on amenable factors. 
For a group action $\Gamma \curvearrowright^\alpha B$ on a von Neumann algebra $B$, W$^*$-superrigidity of $\alpha$ means that the isomorphism class of the action $\alpha$ can be \textit{recovered} from the one of the von Neumann algebra (or the W$^*$-algebra) $B\rtimes_\alpha \Gamma$. 
To be precise, for any action $\Lambda\curvearrowright^\beta A$, if $B\rtimes_\alpha \Gamma\simeq A\rtimes_{\beta} \Lambda$ as von Neumann algebras, then one has $\alpha \simeq \beta$ as actions. 
Here for the action $\beta$, we only assume natural conditions in the framework (e.g.\ free and ergodic action) and do not impose any technical assumptions. 

The first example of W$^*$-superrigid actions was discovered by Popa and Vaes \cite{PV09}. They proved that for a large class of amalgamated free groups, any free ergodic probability measure preserving action is W$^*$-superrigid. 
After this breakthrough work, many examples have been obtained, see  \cite{Pe09,Io10,HPV10,PV11,PV12,Bo12,Io12,Va13,CIK13}. All these works are on actions on probability spaces, namely, actions on commutative von Neumann algebras.

	In the present article, we investigate actions on \textit{amenable factors}. Recall that a von Neumann algebra $M$ (with separable predual) is \textit{amenable} if it is generated by an increasing union of (countably many) finite dimensional von Neumann algebras. 
The amenable von Neumann algebras is the easiest class of von Neumann algebras and contains all commutative von Neumann algebras. 
Hence it is a natural question to ask if a W$^*$-superrigidity phenomena occurs for actions on non-commutative amenable von Neumann algebras. 
However, because of the technical difficulties coming from non-commutativity, none of W$^*$-superrigidity type results for such actions is known so far (even for type II$_1$ factors). 

	We prepare some terminology. We say that a countable discrete group $\Gamma$ is in the \textit{class $\mathcal{C}$} \cite{VV14} if it is non-amenable and for any trace preserving cocycle action $\Gamma\curvearrowright B$ on a finite von Neumann algebra $B$, the following condition holds:
\begin{itemize}
	\item any projection $p \in B\rtimes \Gamma=:M$ and any amenable von Neumann subalgebra $A \subset pMp$, if $A' \cap pMp\subset A$ and if $\mathcal{N}_{pMp}(A)''\subset pMp$ is essentially finite index, then we have $A\preceq_M B$.
\end{itemize}
The class $\mathcal C$ contains all weakly amenable group $\Gamma$ with $\beta_1^{(2)}(\Gamma)>0$ \cite{PV11}, all non-amenable hyperbolic groups \cite{PV12} and all non-amenable free product groups \cite{Io12,Va13}. 
Recall that a faithful normal state $\varphi$ on a von Neumann algebra $M$ is \textit{weakly mixing} if the fixed point algebra of the modular action of $\varphi$ is trivial. In this case $M$ must be a type III$_1$ factor, and the unique amenable type III$_1$ factor admits such a state.

	The following theorem is the main application of Theorem \ref{thmA}. This is the first W$^*$-superrigidity type result for actions on amenable factors. As we will explain below, the proof of this theorem uses the modular theory in a crucial way, and hence cannot be adapted to type II$_1$ factors.
\begin{ThmA}\label{thmB}
	Let $\Gamma$ be an ICC countable discrete group in the class $\mathcal C$, $B_0$ a type $\rm III_1$ amenable factor with separable predual, and $\varphi_0$ a faithful normal state on $B_0$ which is weakly mixing. 
Then the Bernoulli shift action $\Gamma \curvearrowright^\alpha \bigotimes_{\Gamma}(B_0,\varphi_0)(=:(B,\varphi))$ is {\rm W$^*$-superrigid} in the following sense. 

Let $\Lambda\curvearrowright^{\beta} (A,\psi)$ be any state preserving outer action of a discrete group $\Lambda$ on an amenable factor $A$ with a faithful normal state $\psi$. If $B\rtimes_\alpha \Gamma \simeq A\rtimes_\beta \Lambda$, 
then there exist
\begin{itemize}
	\item a finite normal subgroup $\Lambda_0 \leq \Lambda$, so that one has a cocycle action $\Lambda/\Lambda_0 \curvearrowright^{\beta^{\Lambda/\Lambda_0}}(A\rtimes_{\beta}\Lambda_0,\psi')$ by a fixed section $s \colon \Lambda/\Lambda_0 \to \Lambda$, where $\psi'$ is the canonical extension of $\psi$ on $A\rtimes_\beta \Lambda_0$;
	\item a state preserving cocycle action $(\Ad(u_g))_{g\in \Gamma}$ of $\Gamma$ on a type $\rm I$ factor $(\B,\omega)$ equipped with a faithful normal state;
\end{itemize}
such that two actions $\Lambda/\Lambda_0 \curvearrowright^{\beta^{\Lambda/\Lambda_0}}(A\rtimes_{\beta}\Lambda_0,\psi')$ and $\Gamma \curvearrowright^{\alpha\otimes \Ad(u)}(B\ovt \B,\varphi\otimes \omega)$ are conjugate via a state preserving isomorphism.
\end{ThmA}
	The Bernoulli action in this theorem was intensively studied in \cite{VV14,Ve15}. They obtained similar conclusions if the action $\Lambda\curvearrowright^{\beta} (A,\psi)$ is also a Bernoulli action of a group in the class $\mathcal C$.  Now thanks to our Theorem \ref{thmB}, we can put \textit{arbitrary} actions as $\Lambda\curvearrowright^{\beta} (A,\psi)$.

	The conclusion of Theorem \ref{thmB} is optimal. Indeed, 
subgroups and type I factors in the theorem can appear always, since the amenable type III$_1$ factor $B$ has decompositions such as $B=A\rtimes\Lambda_0$ and $B = B \ovt \B$. 
	Note also that the cocycle action $\Lambda/\Lambda_0 \curvearrowright^{\beta^{\Lambda/\Lambda_0}}(A\rtimes_{\beta}\Lambda_0,\psi')$ above depends on the choice of the section $s$, but this dependence affects the cocycle action $\Ad(u)$ on a type I factor only.

	The proof of Theorem \ref{thmB} splits into two steps. 
Firstly, we prove a unique crossed product decomposition theorem: we identify the base algebra $B$ from the von Neumann algebra $B\rtimes_\alpha \Gamma$, so that two actions are cocycle conjugate. Secondly, we prove a cocycle superrigidity type theorem: the corresponding cocycle is cohomologous to a coboundary, so that two actions are conjugate. 

	The next theorem treats the first step. Such a unique crossed product decomposition theorem has been intensively studied during the last decade for actions on finite von Neumann algebras, see \cite{OP07,CS11,PV12,HV12} (and see aforementioned works for W$^*$-superrigidity). Thanks to our Theorem \ref{thmA}, we can put type III factors as base algebras $B$.
\begin{ThmA}\label{thmC}
	Let $\Gamma$ be an ICC countable discrete group in the class $\mathcal C$, $B$ a $\sigma$-finite, amenable, diffuse factor, and $\Gamma \curvearrowright^{\alpha} B$ an outer action. 

Assume that $B\rtimes_\alpha \Gamma \simeq A\rtimes_\beta \Lambda$ for some outer action $\Lambda\curvearrowright^{\beta} A$ of a countable discrete group $\Lambda$ on a $\sigma$-finite, amenable, diffuse factor $A$. 
Then there is an amenable normal subgroup $\Lambda_0 \leq \Lambda$ such that the induced cocycle action $\Lambda/\Lambda_0\curvearrowright^{\beta^{\Lambda/\Lambda_0}} A\rtimes_\beta\Lambda_0$ is cocycle conjugate to  $\alpha$. 
In particular if $\Lambda$ has no amenable normal subgroups, then $\alpha$ and $\beta$ are cocycle conjugate.
\end{ThmA}

	The following immediate corollary generalizes \cite[Theorem 1.10]{PV11}.
\begin{CorA}\label{corD}
	Let $\Gamma \curvearrowright^{\alpha} B$ and  $\Lambda\curvearrowright^{\beta} A$ be outer actions of countable discrete ICC groups on $\sigma$-finite, amenable, diffuse factors such that $B\rtimes_\alpha \Gamma \simeq A\rtimes_\beta \Lambda$. If $\Gamma$ and $\Lambda$ are in the class $\mathcal{C}$, then $\alpha$ and $\beta$ are cocycle conjugate.
\end{CorA}

	We next need a cocycle superrigidity type theorem for the second step. Appropriate adaptations of techniques in \cite{Po05a,Po05b} (see also \cite{VV14,Ma16}) to our setting easily provides the following proposition. This proposition is however \textit{not} useful in our study, as we explain soon below.

\begin{ProA}\label{thmE}
	Let $\Gamma$ be a non-amenable countable discrete group, $(B_0,\varphi_0)$ an amenable factor with separable predual and with a faithful normal state, and $\Gamma \curvearrowright^\alpha \bigotimes_{\Gamma}(B_0,\varphi_0)=:(B,\varphi)$ the Bernoulli shift action. Assume either that $\Gamma$ is a direct product of two infinite groups or has a normal subgroup with relative property (T). 

	Assume that $\alpha$ is cocycle conjugate to some state preserving outer action $\Lambda\curvearrowright^{\beta} (A,\psi)$ of a countable discrete group $\Lambda$ on an amenable factor $A$ with a faithful normal state $\psi$. Then there exists an inner action $(\Ad(u_g))_{g\in \Gamma}$ of $\Gamma$ on a type $\rm I$ factor $\B$ such that two actions $\beta$ and $\alpha\otimes \Ad(u)$ are conjugate.
\end{ProA}

\subsection*{Idea of the proof of Theorem \ref{thmB}}

	We briefly explain the idea of the proof of Theorem \ref{thmB}. The proof uses the modular theory in a crucial way. Consider two actions $\alpha$ and $\beta$ as in Theorem \ref{thmB}.

	Since the group $\Gamma$ is in the class $\mathcal C$, we can first apply Theorem \ref{thmC}. Then an induced \textit{cocycle action} $\beta^{\Lambda/\Lambda_0}$ is cocycle conjugate to $\alpha$. If this cocycle action is a genuine action, by assuming that $\Gamma$ is a direct product or has property (T), one can apply Proposition \ref{thmE} and obtain a conjugacy result. However it is not clear when the cocycle action, which comes from a section $s\colon \Gamma \simeq \Lambda/\Lambda_0 \to \Lambda$, is a genuine action. 
In other words, we do not know when the exact sequence 
$1 \to \Lambda_0 \to \Lambda \to \Gamma \to 1$ splits, where $\Lambda_0$ is amenable and $\Gamma$ is in the class $\mathcal C$ satisfying the assumption of Proposition \ref{thmE}. 
This is the main technical issue to prove the W$^*$-superrigidity theorem in our setting, and this is why such a result is not known even for type II$_1$ factors.

	In the present article, to avoid this problem, we use modular actions. 
Since we assumed that $\alpha$ and $\beta$ are state preserving,  there is an isomorphism
	$$ B\rtimes_{\alpha\times \sigma^\varphi} (\Gamma \times \R) \simeq A\rtimes_{\beta\times \sigma^\psi} (\Lambda \times \R)$$
such that the corresponding (possibly cocycle) actions are cocycle conjugate. 
By assuming that $\varphi_0$ is weakly mixing (which means $\sigma^\varphi$ is weakly mixing), and combining with some rigidity property of Bernoulli actions, one can apply the proof of Proposition \ref{thmE} to the direct product group $\Gamma \times \R$. Here we note that $\R$-actions are always genuine actions, so no technical problems appear in this context. 
Thus the cocycle is cohomologous to a coboundary as $\R$-actions. Since $\R \leq \Gamma \times \R$ is normal and since $\sigma^\varphi$ is weakly mixing, the same conclusion  actually holds as $\Gamma \times \R$-actions and we can finish the proof. 
This is the main idea of the proof of Theorem \ref{thmB}.

\subsection*{Application: stable strong solidity of free product factors}

	The next application is on the structure of amalgamated free product von Neumann algebras. We will generalize Ioana's work \cite{Io12} to the type III setting. 

	Recall that for any (possibly non-unital) inclusions $A,B\subset M$ with expectations and with $1_B=1_M$, we say that \textit{$A$ is injective relative to $B$ in $M$} \cite{OP07,Is17} if there is a conditional expectation $E \colon 1_A \langle M,B\rangle 1_A \to A$ which is faithful and normal on $1_AM1_A$. 
	Recall that for any von Neumann algebra $M$ with the decomposition $M=M_a \oplus M_d$, where $M_a$ is atomic and $M_d$ is diffuse, we say that $M$ is \textit{strongly solid} (resp.\ \textit{stably strongly solid}) \cite{OP07,BHV15} if for any diffuse amenable von Neumann algebra $A \subset M_d$ with expectation, $\mathcal{N}_{M_d}(A)''$ (resp.\ $s\mathcal{N}_{M_d}(A)''$) remains amenable. 
Here $s\mathcal{N}_{M_d}(A)$ is the set of all elements $x\in M_d$ such that $x A x^* \subset A$ and $x^*Ax \subset A$, and such elements are called \textit{stable normalizers}. Then $\mathcal{N}_{M_d}(A)$ is given by $s\mathcal{N}_{M_d}(A) \cap \mathcal{U}(M_d)$ and its elements are called \textit{normalizers}. 
Note that these two notions of strong solidity coincide if $M$ is properly infinite. 
By definition, a strongly solid non-amenable factor $M$ does not admit any crossed product decomposition $M=A\rtimes \Gamma$ (for amenable $A$), so strong solidity should be understood as a strong \textit{indecomposability} of $M$.

	The following theorem is a generalization of Ioana's theorem \cite[Theorem 1.6]{Io12} (see also \cite{Va13,HU15,BHV15}). As a corollary, we characterize stable strong solidity of free product factors, see \cite[Theorem 1.8]{Io12} for the same characterization for type II$_1$ factors.
\begin{ThmA}\label{thmF}
	Let $B \subset M_i$ be inclusions of $\sigma$-finite von Neumann algebras with expectations $E_i$ for $i=1,2$. Let $M:=(M_1,E_1)*_B(M_2 ,E_2)$ be the amalgamated free product von Neumann algebra, $p\in M$ a projection, and $A\subset pMp$ a von Neumann subalgebra with expectation. Assume that $A$ is injective relative to $B$ in $M$ and assume that $A'\cap pMp \subset A$. Then at least one of the following conditions holds true:
\begin{enumerate}
	\item[$\rm (i)$] $A\preceq_M B$;
	\item[$\rm (ii)$] $s\mathcal{N}_{pMp}(A)''\preceq_M M_i$ for some $i\in \{1,2\}$;
	\item[$\rm (iii)$] $s\mathcal{N}_{pMp}(A)''$ is injective relative to $B$.
\end{enumerate}
\end{ThmA}
\begin{CorA}\label{corG}
	Let $I$ be a set and $(M_i,\varphi_i)_{i\in I}$ a family of nontrivial von Neumann algebras with faithful normal states. Put $M:=*_{i\in I}(M_i,\varphi_i)$. 
Then $M$ is stably strongly solid if and only if so are all $M_i$'s.
\end{CorA}

	Examples of stably strongly solid factors have been obtained in several articles \cite{BHV15,BDV17,Ma18,HT18}. Also all amenable von Neumann algebras are stably strongly solid.  Using these algebras,  Corollary \ref{corG} provides plenty of new examples of stably strongly solid factors.

\bigskip

\noindent 
{\bf Acknowledgement.} The author would like to thank Cyril Houdayer, Amine Marrakchi, and Stefaan Vaes for many useful comments on the first draft of this manuscript. He also would like to thank Yuki Arano and Toshihiko Masuda for fruitful conversations on group actions on factors. 
He was supported by JSPS, Research Fellow of the Japan Society for the Promotion of Science.

\tableofcontents

\section{Preliminaries}\label{Preliminaries}

\subsection*{Tomita--Takesaki theory}

	Let $M$ be a von Neumann algebra and $\varphi$ a faithful normal semifinite weight on $M$. Throughout the paper, for objects in Tomita--Takesaki's modular theory, we will use the following notation. 
The \textit{modular operator, conjugation}, and \textit{action} are denoted by $\Delta_\varphi$, $J_\varphi$, and $\sigma^\varphi$ respectively. The \textit{continuous core}, which is the crossed product von Neumann algebra $M\rtimes_{\sigma^\varphi}\R$, 
 is denoted by $C_\varphi(M)$, and $\Tr_\varphi$ and $L_\varphi \R$ mean the canonical trace on $C_\varphi(M)$ and the canonical copy of $L\R$ in $C_\varphi(M)$ respectively. 
The \textit{centralizer algebra} $M_\varphi$ is a fixed point algebra of the modular action. The norm $\|\, \cdot \, \|_\infty$ is the operator norm of $M$, while $\|\, \cdot \, \|_{2,\varphi}$ (or $\| \, \cdot \, \|_{\varphi}$) is the $L^2$-norm by $\varphi$. See \cite{Ta03} for definitions of all these objects.

	For any continuous action $G \curvearrowright^\alpha M$ of a locally compact group $G$, in this article, we will use the following canonical embeddings for crossed products: $\pi_\alpha\colon M \to M\rtimes_\alpha G$ by $(\pi_\alpha(x)\xi)(g)=\alpha_{g^{-1}}(x)\xi(g)$ for all $\xi \in L^2(G,L^2(M))$ and $g\in G$; and $G \to M\rtimes_\alpha G$ by $g \mapsto 1_M\otimes \lambda_g$ for all $g\in G$. 
Via these embeddings, we often regard $M$ and $LG$ as subalgebras of $M\rtimes_\alpha G$.

\subsection*{Connes cocycle}

	Let $G$ be a locally compact group, $M$ a von Neumann algebra and $G\curvearrowright^\alpha M$ a continuous action (see \cite[Definition X.1.1]{Ta03} for continuity). Let $p \in M$ be a nonzero projection. We say that a $\sigma$-strongly continuous map $u\colon G \to pM$ is a \textit{generalized cocycle for $\alpha$ (with support projection $p$)} if 
\begin{itemize}
	\item $u_{gh}=u_g \alpha_g(u_h)$ for all $g,h\in G$;
	\item $u_gu_g^* = p$, \quad $u_g^* u_g = \alpha_g(p)$ for all $g\in G$.
\end{itemize}
In this case, by putting $\alpha^u_g(pxp):=u_g\alpha_g(pxp)u_g^*$ for all $x\in M$ and $g\in G$, one has a continuous $G$-action on $pMp$. It holds that $p(M \rtimes_\alpha G)p \simeq pMp \rtimes_{\alpha^u}G $. When $p=1$, we simply say that $u$ is a \textit{cocycle}.

	Let $N$ be another von Neumann algebra and consider continuous actions $G\curvearrowright^\alpha M$ and $G\curvearrowright^\beta N$. We say that they are \textit{$\alpha$ is cocycle conjugate to $\beta$ via a generalized cocycle} if there exist a projection $p\in M$, a $\ast$-isomorphism $\pi\colon pMp \to N$ and a generalized cocycle $u\colon G \to pM$ for $\alpha$ with support projection $p$ such that 
	$$ \pi^{-1}\circ\beta_g\circ \pi (a)  = u_g\alpha_g(a) u_g^*, \quad \text{for all }a\in pMp, \ g\in G.$$
In this case, by identifying $pMp=N$ by $\pi$, we can define a partial isometry $U \colon L^2(G,L^2(M)) \to L^2(G,L^2(M)) $ by $(U\xi)(g) =u_{g^{-1}} \xi(g) = p u_{g^{-1}} \alpha_{g^{-1}}(p) \, \xi(g)$ for $g\in G$. 
Note that $U^* U = \pi_\alpha(p)$ and $UU^* = p\otimes 1_{L^2(G)}$. 
One has a $\ast$-isomorphism
	$$ \Pi_{\beta,\alpha}:=\Ad(U) \colon p(M\rtimes_\alpha G)p \to pMp \rtimes_{\beta} G $$
satisfying $\Pi_{\beta,\alpha}(x)=x$ for $x\in pMp$ and $\Pi_{\beta,\alpha}(p\lambda^\alpha_gp) = pu_g\lambda_g^\beta p=u_g\lambda_g^\beta$ for $g\in G$. 
If one can choose $p=1$, so that $u$ is a cocycle, then we simply say that $\alpha$ and $\beta$ are \textit{cocycle conjugate}.

	Let $M$ be a von Neumann algebra and $\varphi,\psi$ normal semifinite weights on $M$. Assume that $\varphi$ is faithful and let $s(\psi)$ be the support projection of $\psi$. 
Consider modular actions $\sigma^\varphi$ on $M$ and $\sigma^\psi$ on $s(\psi)Ms(\psi)$. 
The \textit{Connes cocycle} $([D\psi,D\varphi]_t)_{t\in \R}$ \cite{Co72} is a generalized cocycle for $\sigma^\varphi$ with support projection $s(\psi)$ such that $\sigma^\varphi$ is cocycle conjugate to $\sigma^\psi$ via  $([D\psi,D\varphi]_t)_{t\in \R}$. 
In particular, there is a canonical $\ast$-isomorphism
	$$ \Pi_{\psi,\varphi} \colon pC_\varphi(M)p= p(M\rtimes_{\sigma^\varphi} G)p \to pMp \rtimes_{\sigma^\psi} G= C_\psi(pMp) .$$
See \cite[V.III.3.19-20]{Ta03} for this non-faithful version of the Connes cocycle. 
In this article, we need the following important theorem. 

\begin{Thm}[{\cite[TH$\rm \acute{E}$OR$\rm \grave{E}$ME 1.2.4]{Co72}}]\label{connes cocycle existence}
	Let $M$ be a von Neumann algebra and $\varphi$ a faithful normal semifinite weight on $M$. 
Let $p\in M$ be a projection and $(u_t)_{t\in \R}$ is a generalized cocycle for $(\sigma_t^\varphi)_t$ with support projection $p$. 
Then there is a unique normal semifinite weight $\psi$ on $M$ such that $s(\psi)=p$ and $u_t = [D\psi:D\varphi]_t$ for all $t\in \R$.
\end{Thm}

Below, we record an elementary lemma. 
We use the notation $x\varphi y = \varphi(y \, \cdot \, x)$.

\begin{Lem}\label{connes cocycle lemma}
	Let $M$ be a von Neumann algebra and $\varphi,\psi\in M_*$ faithful positive functionals. 
\begin{enumerate}
	\item[$\rm (1)$] For any projection $e\in M_\psi$, we have
	$$ [De\psi e, D\psi]_t=e \quad \text{and} \quad  e[D\psi , D\varphi]_t = [De\psi e, D\varphi]_t .$$ 
In particular we have a chain rule: 
	$$ [De\psi e, D\psi]_t \, [D\psi , D\varphi]_t = [De\psi e, D\varphi]_t .$$
	\item[$\rm (2)$] Let $v \in M$ be a partial isometry such that $e:=v v^*\in M_\psi$ and $f:=v^*v \in M_\varphi$. 
Assume that $v\varphi v^* = e\psi e$ on $M$ (equivalently $f\varphi f = v^*\psi v$). Then we have 
	$$v\sigma_t^\varphi(v^*xv)v^*=\sigma_t^\psi(exe), \quad v^*[D\psi , D\varphi]_t =  \sigma_t^\varphi(v^*), \quad x\in M,\ t\in \R.$$
\end{enumerate}
\end{Lem}

\subsection*{Cocycle actions}

	A more general notion of a group action is a cocycle action. We say that a locally compact group $G$ acts on a von Neumann algebra $M$ as a \textit{cocycle action} if there exist continuous maps $\alpha\colon G \to \Aut(M)$ and $v\colon G \times G \to \mathcal{U}(M)$ such that 
\begin{align*}
	& \alpha_e =\id, \quad \alpha_g\circ \alpha_h = \Ad(v(g,h))\circ \alpha_{gh}, \quad v(g,h)v(gh,k) = \alpha_{g}(v(h,k))v(g,hk)
\end{align*}
for all $g,h,k \in G$, where $e$ is the neutral element. The map $v$ is called a \textit{2-cocycle}. 
Two cocycle actions $G\curvearrowright^{(\alpha,v)}M$ and $G\curvearrowright^{(\beta,w)}N$ are said to be \textit{cocycle conjugate} if there exist a $\ast$-isomorphism $\pi \colon M \to N$ and  a continuous map $u \colon G \to \mathcal{U}(M)$ such that, for all $g,h\in G$,
\begin{align*}
	 \pi^{-1}\circ \beta_g \circ \pi=\Ad(u_g)\circ \alpha_g , \quad \pi^{-1}(w(g,h)) = u_g \alpha_g(u_h)v(g,h)u^*_{gh}.
\end{align*}
	In this article, cocycle actions appear in the following two contexts. 

	Let $\Gamma\curvearrowright^\alpha B$ be an action of a discrete group on a von Neumann algebra $B$. Let $p\in B$ be a projection and assume that $\alpha_g(p)\sim p$ in $B$ for all $g\in G$. Take any partial isometries $w_g\in B$ such that $w_gw_g^*=p$ and $w_g^*w_g=\alpha_g(p)$ for all $g\in \Gamma$. 
Define $\alpha_g^p(x):=w_g\alpha_g(x)w_g^*$ and $v^p(g,h):=w_g\alpha_g(w_h)w_{gh}^*$ for all $x\in pBp$, $g,h\in \Gamma$. Then $(\alpha^p,v^p)$ is a cocycle action on $pBp$ satisfying $p(B\rtimes_\alpha \Gamma)p \simeq pBp \rtimes_{(\alpha^p,v^p)}\Gamma$.

	Let $\Gamma \curvearrowright^\alpha B$ be the same group action. Let $\Lambda \leq \Gamma$ be a normal subgroup and fix a section $s\colon \Gamma /\Lambda \to \Gamma$ such that $s(\Lambda)$ is the unit of $\Gamma$. 
Inside $B\rtimes_\alpha \Gamma$, for all $g,h \in \Gamma/\Lambda$, we define 
	$${\alpha}^{\Gamma/\Lambda}_{g} := \Ad(\lambda^{\Gamma}_{s(g)}) \in \mathrm{Aut}(B\rtimes_\alpha \Lambda),\quad \text{and} \quad
 v(g,h):=\lambda^{\Gamma}_{ s(g)s(h)s(gh)^{-1} }\in L\Lambda.$$
It is easy to verify that $\alpha^{\Gamma/\Lambda}$ and $v$ define a cocycle action of $\Gamma/\Lambda$ on $B\rtimes_\alpha \Lambda$ satisfying $B\rtimes_\alpha \Gamma \simeq (B\rtimes_\alpha \Lambda)\rtimes_{(\alpha^{\Gamma/\Lambda},v)}\Gamma/\Lambda$.

\subsection*{Basic constructions and operator valued weights}

	For operator valued weights, we refer the reader to \cite{Ha77a,Ha77b}. We will say that a unital inclusion $B\subset M$ of von Neumann algebras is \textit{with operator valued weight} if there is an operator valued weight $E_B \colon M \to B$.

	Let $B \subset M$ be a unital inclusion of $\sigma$-finite von Neumann algebras with expectation $E_B$. Fix a faithful normal state $\varphi$ on $M$ such that $\varphi=\varphi\circ E_B$. 
Put $L^2(M):=L^2(M,\varphi)$ and $J:=J_\varphi$, and consider $B\subset M \subset \B(L^2(M))$. 
The von Neumann algebra $ \langle M,B\rangle := (JBJ)' $ is called the \textit{basic construction}, and is generated by $Me_B M$, where $e_B$ is the Jones projection for $E_B$. 
Using the inclusion $JBJ \subset JMJ$ with expectation $JE_BJ := \Ad(J) \circ E_B \circ \Ad(J)$, one can define a canonical operator valued weight $(JE_BJ)^{-1}\colon (JBJ)' \to (JMJ)=M$. We will write as $\widehat{E}_B:=(JE_BJ)^{-1}$. It satisfies that $\widehat{E}_B(b^* e_Ba) = b^*a$ for all $a,b \in M$. 
See \cite{Ko85,ILP96} for the general theory of $\widehat{E}_B$. 

	Below we collect well known facts for basic constructions and operator valued weights, which we will need in this article.
\begin{itemize}
	\item For any faithful $\psi\in M_*^+$, one can define a faithful normal semifinite weight $\widehat{\psi}:=\psi \circ \widehat{E}_B$ on $\langle M,B\rangle$. It holds that 
	$$\sigma_t^{\widehat{\psi}}|_M = \sigma_t^\psi \quad \text{and} \quad [D\widehat{\psi}:D\widehat{\varphi}]_t = [D\psi:D\varphi]_t \quad  \text{
for all }t\in \R.$$

	\item Let $E_{C_\varphi(B)}\colon C_\varphi(M) \to C_\varphi(B)$ be the canonical conditional expectation such that $E_{C_\varphi(B)}|_M=E_B$ and $E_{C_\varphi(B)}|_{L_\varphi \R} = \id$.  
Using $\sigma^{{\varphi}}_t\circ \widehat{E}_B  = \widehat{E}_B \circ \sigma^{\widehat{\varphi}}_{t}$ for all $t\in \R$, one can define an operator valued weight from $\langle M,B\rangle \rtimes_{\sigma^{\widehat{\varphi}}} \R$ to $M \rtimes_{\sigma^\varphi} \R$ whose restriction on $\langle M,B\rangle^+$ coincides with $\widehat{E}_B$. 
We will denote it by $\widehat{E}_B \rtimes \R$.

	\item We canonically have 
	$$\langle C_\varphi(M),C_\varphi(B) \rangle = C_{\widehat{\varphi}}(\langle M,B\rangle).$$
The left hand side has a canonical operator valued weight $\widehat{E}_{C_\varphi(B)}$ onto $C_\varphi(M)$, and the right hand side has $\widehat{E}_B \rtimes \R$. 
Since constructions are canonical, these two operator valued weights coincide. 
\end{itemize}

Here we prove a lemma for type III$_1$ factors. 
\begin{Lem}\label{III1 factor tensor lemma}
	Let $A\subset M$ be a unital inclusion of von Neumann algebras with an operator valued weight $E_A$. Fix a faithful $\psi_A\in A_*^+$, and put ${\psi}:=\psi_A \circ E_A$. 
Let $N$ be a type $\rm III_1$ factor with a faithful normal semifinite weight $\omega$. 
Then the following equation holds true:
	$$ C_{\psi \otimes \omega}(A\ovt N)' \cap C_{\psi\otimes \omega}(M\ovt N) = \left( A'\cap M_\psi \right) \otimes \C1_N \otimes \C 1_{L^2(\R)}.$$
\end{Lem}
\begin{proof}
	Since $N$ is a type III$_1$ factor, there is a faithful normal semifinite weight $\omega'$ such that $(N_{\omega'})' \cap N=\C$ (see \cite[Theorem XII.1.7]{Ta03}). Thanks to the Connes cocycle, there is a canonical isomorphism from $C_{\psi\otimes \omega'}(M\ovt N)$ to $C_{\psi\otimes \omega}(M\ovt N)$ which sends $C_{\psi\otimes \omega'}(A\ovt N)$ onto $C_{\psi\otimes \omega}(A\ovt N)$ and which is the identity on $M\ovt N$. Hence to prove this lemma, by exchanging $\omega'$ with $\omega$, we may assume that $N_{\omega}' \cap N=\C$.

For simplicity we write as $ L_{\psi\otimes\omega}\R=L\R $. 
Observe that (e.g.\ \cite[Proposition 2.4]{HR10})
\begin{align*}
	C_{\psi\otimes \omega}(\C 1_{{A}}\otimes \C 1_N)' \cap  C_{\psi\otimes \omega}(M\ovt N)	\subset & \ (M\ovt N)_{{\psi\otimes \omega}} \ovt L\R .
\end{align*}
Since $(\C 1_A \otimes N_\omega )' \cap (M\ovt N)_{{\psi\otimes \omega}}= M_{{\psi}}\otimes \C1_N$, we have 
\begin{align*}
	C_{\psi\otimes\omega}(\C1_A\otimes N_\omega)' \cap  C_{\psi\otimes \omega}(M\ovt N)
	\subset & \ M_{{\psi}}\ovt \C1_N \ovt L\R.
\end{align*}
Since $C_\omega(N)$ is a factor, it holds that $\pi_\omega(N)' \cap (\C1_N\otimes L_{\omega} \R) =\C1_N \otimes \C1_{L^2(\R)}$, where $\pi_\omega(N)$ is the canonical image of $N$ in $C_\omega(N)$. This implies that
\begin{align*}
	C_{\psi\otimes\omega}(\C1_A\otimes N)' \cap  C_{\psi\otimes \omega}(M\ovt N)
	\subset & \ M_{\psi}\ovt\left[ \pi_\omega(N)' \cap (\C1_N\otimes L \R)\right]\\
	= & \ M_{\psi}\ovt \C 1_N \ovt \C1_{L^2(\R)}. 
\end{align*}
Using the canonical embedding $\pi_{\psi \otimes \omega}$, the last term coincides with $\pi_{\psi \otimes \omega}(M_\psi \otimes \C 1_N)$, hence 
\begin{align*}
	C_{\psi \otimes \omega}(A\ovt N)' \cap C_{\psi\otimes \omega}(M\ovt N)
	&=\pi_{\psi\otimes \omega}(A\ovt \C 1_N)'\cap  \pi_{\psi\otimes \omega}(M_\psi \otimes \C 1_N)\\
	&=\pi_{{\psi\otimes \omega}}(\left(A' \cap M_{{\psi}}\right) \otimes \C1_N) \\
	&= \left( A'\cap M_{{\psi}} \right) \otimes \C1_N \otimes \C 1_{L^2(\R)}. 
\end{align*}
This is the conclusion.
\end{proof}

\subsection*{Popa's intertwining theory}

	As explained in Section \ref{Introduction}, we refer the reader to \cite{Po01,Po03} for the origin of intertwining theory. Here we give a definition introduced in \cite{HI15}. 

\begin{Def}\label{def corner intertwining}\upshape
	Let $M$ be a $\sigma$-finite von Neumann algebra and $A,B\subset M$ (possibly non-unital) von Neumann subalgebras with expectation. 
We will say that \textit{a corner of $A$ embeds with expectation into $B$ inside $M$} and write $A \preceq_M B$ if there exist  projections $e\in A$, $f\in B$, a partial isometry $v\in eMf$ and a unital normal $\ast$-homomorphism $\theta\colon eAe \to fBf$ such that 
		\begin{itemize}
			\item $\theta(eAe) \subset fBf$ is with expectation;
			\item $v\theta(a)= av$ for all $a\in eAe$.
		\end{itemize}
In this case, we will say that \textit{$(e,f,\theta,v)$ witnesses $A\preceq_MB$}.
\end{Def}

	We recall known characterizations of the intertwining condition $A\preceq_MB$.  
For this, we borrow notation from \cite{HI15}. We refer the reader to \cite[Section 4]{HI15} for items here. The same notation will be used in Section \ref{Popa's intertwining techniques with conditional expectations}.

	Let $M$ be a $\sigma$-finite von Neumann algebra and $A,B\subset M$ (possibly non-unital) von Neumann subalgebras with expectations. Fix a faithful normal conditional expectation $E_B$ for $B\subset 1_BM1_B$. 
Put $\widetilde{B}:=B\oplus \C(1_M-1_B)$ and let $E_{\widetilde{B}}\colon M\to \widetilde{B}$ be a faithful normal conditional expectation which extends $E_B$. 
	Let $B = B_1 \oplus B_2 $ be the unique decomposition such that $B_1$ is finite and $B_2$ is properly infinite. Fix a faithful normal  trace $\tau_{B_1}$ on $B_1$ and choose a faithful normal state $\varphi\in M_*$ such that $\varphi$ is preserved by $E_B$ and $E_{\widetilde{B}}$ and that $\varphi|_{B_1} = \tau_{B_1}$ (up to scalar multiples). 
Fix a standard representation $L^2(M):=L^2(M,\varphi)$ and its modular conjugation $J:=J_\varphi$. 
We write as $e_{\widetilde{B}}$ and $e_B$ corresponding Jones projections (note that $e_{\widetilde{B}}1_B=e_{\widetilde{B}}J1_BJ =e_B$), and as $\widehat{E}_{\widetilde{B}}$ the canonical operator valued weight from $\langle M,\widetilde{B}\rangle$ to $M$ given by $\widehat{E}_{\widetilde{B}}(xe_{\widetilde{B}}x^*)=xx^*$ for all $x\in M$. 
Denote by $\Tr$ the unique trace on $\langle M,\widetilde{B}\rangle J1_{B_1}J$ satisfying $\Tr( (x^*e_{\widetilde B}x) J1_{B_1}J) =  \tau_{B_1}( E_B(1_{B_1}xx^*1_{B_1}))$ for all $x \in M$. 
Since $\mathcal{Z}(\langle M,\widetilde{B}\rangle J1_{B_1}J ) = J\mathcal{Z}(B_1)J$, there is a unique operator valued weight $\mathrm{ctr}\colon \langle M,\widetilde{B}\rangle J1_{B_1}J \to J\mathcal{Z}(B_1)J$ such that $\Tr = \overline{\tau_{B_1}(J\, \cdot \, J)} \circ \mathrm{ctr}$. Since $\Tr$ is a trace, $\mathrm{ctr}$ is an extended center valued trace. 
Let $\mathrm{ctr}_{B_1}$ be the center valued trace for $B_1$ and recall that $\tau_{B_1} \circ \mathrm{ctr}_{B_1} = \tau_{B_1}$. 
It holds that 
	$$\mathrm{ctr}((x^*e_{\widetilde B}x) J1_{B_1}J) = J\mathrm{ctr}_{B_1}\circ E_B(1_{B_1}xx^*1_{B_1})J, \quad \text{for all }x\in M.$$
We mention that the decomposition $B=B_1\oplus B_2$ here is slightly different from the one in \cite{HI15}, and that $\mathrm{ctr}$ was not used in \cite{HI15}. However the proof of \cite[Theorem 4.3]{HI15} works without any change if we use $\mathrm{ctr}$ and our decomposition for $B$. Our items introduced here are more appropriate in the context of intertwining conditions with actions, which will be discussed in the next section.

	Now we introduce Popa's intertwining theorem. We refer the reader to \cite[Theorem 4.3]{HI15} and \cite[Theorem 2]{BH16} for the proof of this version.

\begin{Thm}\label{thm corner intertwining}
The following conditions are equivalent.
	\begin{itemize}
		\item[$\rm (1)$] We have $A \preceq_M B$.
		\item[$\rm (2)$] There exists a nonzero positive element $d\in A' \cap 1_A\langle M,\widetilde{B}\rangle 1_A$ such that 
	$$d=d J1_BJ \quad \text{and} \quad \widehat{E}_{\widetilde{B}}(d)\in M.$$
\end{itemize}
If $A$ is finite, then the following condition is also equivalent.
\begin{itemize}
		\item[$\rm (3)$] There is no net $(u_i)_i$ in $\mathcal{U}(A)$ such that $E_B(b^* u_ia) \to 0$ $\sigma$-strongly for all $a,b\in M1_B$.
\end{itemize}
\end{Thm}

	Using the next lemma, we can exchange the map $\theta$ for the condition $A\preceq_MB$ with a unital $\ast$-homomorphism on $A$.

\begin{Lem}\label{relation for unital and corner}
	The following assertions hold true.
\begin{itemize}
	\item[$\rm (1)$] The condition $A\preceq_MB$ is equivalent to the following condition: there exist a separable Hilbert space $H$, a projection $f\in B\ovt \B(H)$, a partial isometry $w\in (1_A\otimes e_{1,1})(M\ovt \B(H))f$, where $e_{1,1}$ is a minimal projection, and a unital normal $\ast$-homomorphism $\pi\colon A \to f(B\ovt \B(H))f$ such that 
		\begin{itemize}
			\item $\pi(A) \subset f(B \ovt \B(H))f$ is with expectation;
			\item $w\pi(a)= (a\otimes e_{1,1}) w$ for all $a\in A$.
		\end{itemize}
In this case, (to distinguish $A\preceq_MB$,) we will say that {\rm $(H,f,\pi,w)$ witnesses $A\preceq^{\rm uni}_MB$}.

	\item[$\rm (2)$] 
	Assume either one of the following conditions holds: 
\begin{itemize}
	\item $A$ does not have any direct summand which is semifinite and properly infinite; or
	\item $B$ is properly infinite.  
\end{itemize}
If $A\preceq_MB$ holds, then the Hilbert space $H$ in item $(1)$ can be taken as finite dimensional. 
\end{itemize}
\end{Lem}
\begin{proof}
	Since we will prove a very similar but a more complicated statement in Lemma \ref{relation for unital and corner of actions}, we omit the proof. 
Indeed, to prove this lemma, one can follow the proof of Lemma \ref{relation for unital and corner of actions} by regarding actions are trivial (and by using \cite[Theorem 4.3 and Lemma 4.10]{HI15}).
\end{proof}

\section{Intertwining theory with modular actions}\label{Popa's intertwining techniques with conditional expectations}

	In this section, we introduce several variants of Popa's intertwining condition. We investigate these conditions as well as relations between them. At the end of this section, we prove Theorem \ref{thmA}.
	Throughout this section, we always fix (possibly non-unital) inclusions $A,B \subset M$ of $\sigma$-finite von Neumann algebras with expectations $E_A,E_B$ respectively.

\subsection*{Intertwining theory with group actions}

	We first consider the intertwining condition $A\preceq_MB$ when a locally compact group acts on them. This idea was first used in \cite{Po04,Po05a} to study cocycle superrigidity for discrete group actions. Although our main interest is the case of modular actions, we first study this condition by assuming that a general locally compact group acts on $A,B \subset M$.

	We fix the following setting (which will be used in Definitions \ref{def unital intertwining with actions} and Theorem \ref{thm unital intertwining with actions}). 
We use notation introduced before Theorems \ref{thm corner intertwining}, so we use $A\subset 1_AM1_A$, $B \subset 1_BM1_B$, $B=B_1\oplus B_2$, $\widetilde{B}$, $E_B$, $E_{\widetilde{B}}$, $L^2(M)$, $\varphi$, $J$, $e_B$, $e_{\widetilde{B}}$, $\tau_{B_1}$, $\Tr$, $\widehat{E}_{\widetilde{B}}$, and $\mathrm{ctr}$. 
Let $G$ be a locally compact second countable group, and consider continuous actions $\alpha$ and $\beta$ of $G$ on $M$ such that 
\begin{itemize}
	\item $\alpha_g(A) =A$ and $\beta_g(B) = B$ for all $g\in G$;
	\item $\alpha_g\circ E_A = E_A \circ \alpha_g$ on $1_AM1_A$ and $\beta_g\circ E_B = E_B\circ \beta_g$ on $1_BM1_B$ for all $g\in G$;
	\item $\alpha$ and $\beta$ are cocycle conjugate: there exists 
a $\beta$-cocycle $\omega \colon G \to M$ such that $\alpha_g=\Ad(\omega_g) \circ \beta_g (=:\beta_g^\omega)$ for all $g\in G$.
\end{itemize}
In this setting, based on the viewpoint of Lemma \ref{relation for unital and corner}(1), we define intertwining conditions \textit{with group actions} as follows. 
\begin{Def}\label{def unital intertwining with actions}\upshape
	Keep the setting. We say that $(A,\alpha)$ \textit{embeds with expectation into $(B,\beta)$ inside $M$} and write $(A,\alpha) \preceq^{\rm uni}_{M} (B,\beta)$ if there exist: $(H,f,\pi,w)$ which witnesses $A\preceq^{\rm uni}_MB$ (in the sense of Lemma \ref{relation for unital and corner}(1)), and a generalized cocycle $(u_g)_{g\in G}$ for $\beta\otimes \id_H$ with values in $B\ovt \B(H)$ and with support projection $f$ such that 
		\begin{itemize}
			\item $wu_g =  (\omega_g\otimes 1_H )(\beta_g\otimes \id_H) (w)$ for all $g\in G$;
			\item $u_g(\beta_g\otimes \id_H)(\pi(a))u_g^* = \pi(\alpha_g(a))$ for all $g\in G$ and $a\in A$.
		\end{itemize}
In this case, we will say that \textit{$(H,f,\pi,w)$ and $(u_g)_{g\in G}$ witness $(A,\alpha) \preceq^{\rm uni}_{M} (B,\beta)$}.
\end{Def}
	Before proceeding, we mention following remarks. 
\begin{itemize}
	\item In the definition, using the polar decomposition, $w$ is not necessarily a partial isometry (e.g.\ \cite[Remark 4.2(1)]{HI15}). 

	\item We can define a $\ast$-isomorphism  $\Pi^{\omega}_{\beta,\alpha}\colon M\rtimes_{\alpha}G \to M\rtimes_\beta G$ such that $\Pi^{\omega}_{\beta,\alpha}(a)=a$ for $a\in M$ and $\Pi^{\omega}_{\beta,\alpha}(\lambda^\alpha_g) = \omega_g \lambda_g^\beta$ for $g\in G$. 
There exist unital inclusions $A\rtimes_\alpha G \subset 1_A(M\rtimes_\alpha G)1_A$ and $B\rtimes_\beta G \subset 1_B(M\rtimes_\beta G)1_B$. 

	\item Using compression maps by $e_B\otimes 1$ and $e_A \otimes 1$, faithful normal conditional expectations $E_{B\rtimes_\beta G}\colon 1_B(M\rtimes_\beta G)1_B \to B\rtimes_\beta G$ and $E_{A\rtimes_\alpha G}\colon 1_A(M\rtimes_\alpha G)1_A \to A\rtimes_\alpha G$ are defined. 

	\item For each $g\in G$, let $u_g^\beta \in \mathcal{U}(L^2(M))$ be the canonical implementing unitary for $\beta_g$. Then putting $\widehat{\beta}_g:=\Ad(u_g^\beta)$, the action $\beta$ can be extended on $\langle M, \widetilde{B}\rangle$. 

	\item Putting $\widehat{\alpha}_g:=\Ad(\omega_g u_g^\beta)=\Ad(\omega_g)\circ \widehat{\beta}$ for $g\in G$, we can also extend $\alpha$ on $\langle M, \widetilde{B}\rangle$. 
Note that $\widehat{\alpha}_g(1_A) = 1_A$ and $\widehat{\alpha}_g(J1_BJ) = J1_{B}J$ for all $g\in G$.

	\item For each $g\in G$, since $\beta_g $ commutes with $E_B$, it holds that $\widehat{E}_{\widetilde{B}}\circ \widehat{\beta}_g = \beta_g \circ \widehat{E}_{\widetilde{B}}$ on $(\langle M,\widetilde{B} \rangle J1_BJ)^+$. This implies that $\widehat{E}_{\widetilde{B}}\circ \widehat{\alpha}_g = \alpha_g \circ \widehat{E}_{\widetilde{B}}$ on $(\langle M,\widetilde{B} \rangle J1_BJ)^+$.
\end{itemize}

	Our first goal in this section is to prove the following theorem, which gives fundamental characterizations of the condition $(A,\alpha) \preceq_M (B,\beta)$. 
We mention the origins of these conditions can be found in \cite{Po04,Po05a} (see also \cite{HSV16}).

\begin{Thm}\label{thm unital intertwining with actions}
Consider the following conditions.
	\begin{itemize}
		\item[$\rm (1)$] We have $(A,\alpha) \preceq^{\rm uni}_M (B,\beta)$.
		\item[$\rm (2)$] We have $\Pi^{\omega}_{\beta,\alpha}(A\rtimes_\alpha G) \preceq_{M\rtimes_\beta G} B\rtimes_\beta G$. 
		\item[$\rm (3)$] There exists no nets $(u_i)_{i}$ of unitaries in $\mathcal U(A)$ and $(g_i)_i$ in $G$ such that 
	$$E_B(\beta_{g_i}(b^*)\omega_{g_i}^* u_i a ) \rightarrow 0, \quad \text{$\sigma$-strongly for all $a,b\in M1_B$.}$$
		\item[$\rm (4)$] There exists a nonzero positive element $d\in A' \cap 1_A\langle M, \widetilde{B}\rangle^{\widehat{\alpha}}1_A$ such that 
	$$d=dJ1_{B}J \quad \text{and} \quad \widehat{E}_{\widetilde{B}}(d)\in M.$$
\end{itemize}
Then we have $(4) \Leftrightarrow (1) \Rightarrow (2)$. Moreover the following assertion holds true.
\begin{itemize}
	\item Assume further that $A\rtimes_\alpha G$ is finite. Then we have $(2) \Leftrightarrow (3) \Rightarrow (4)$, hence all conditions are equivalent. In this case, we can choose a Hilbert space $H$ in item $(1)$ as finite dimensional. 
\end{itemize}
\end{Thm}
\begin{Rem}\upshape
	In the case $A=\C$, combining with Theorem \ref{relation for expectations and actions} below, this theorem generalizes \cite[Theorem 3.1]{HSV16}. 
When $A$ is not finite, the theorem fails since there is a counterexample for the implication (2)$\Rightarrow$(1) by \cite[Theorem 4.9]{HI17}. 
We will nevertheless use this theorem by taking tensor products with a type III$_1$ factor, see Lemma \ref{key lemma}. 
\end{Rem}
\begin{proof}
Throughout the proof, we will write a tensor product with $\B(H)$ as with the symbol $H$ at the top, such as $M^H := M\ovt\B(H)$, $\alpha^H_g:=\alpha_g\otimes \id_{H}$, $\omega^H_g:=\omega_g\otimes 1_{H}$ etc. 

$(1) \Rightarrow (2)$ 
	Fix $(H,f,\pi,w)$ and $(u_g)_{g\in G}$. The generalized cocycle $(u_g)_{g\in G}$ gives a $\ast$-isomorphism 
	$$\Pi^u_{\beta^H, (\beta^H)^u}\colon f(M^H\rtimes_{(\beta^H)^u}G)f \to f(M^H\rtimes_{\beta^H} G)f$$ 
satisfying $\Pi^u_{\beta^H, (\beta^H)^u}(faf)=faf$ for $a\in M^H$ and $\Pi^u_{\beta^H, (\beta^H)^u}(f\lambda_g^{(\beta^H)^u}f) = fu_g \lambda_g^{\beta^H} f = u_g \lambda_g^{\beta^H}$ for $g\in G$. Note that this restricts to a $\ast$-isomorphism between $f(B^H\rtimes_{(\beta^H)^u}G)f$ and $f(B^H\rtimes_{\beta^H} G)f$. 
The equivariance property $(\beta^H)^u_g(\pi(a))=u_g\beta^H_g(\pi(a))u_g^* = \pi(\alpha_g(a))$ for $a\in A$ and $g\in G$ implies that there is a $\ast$-homomorphism 
	$$A\rtimes_\alpha G \to \pi(A)\rtimes_{(\beta^H)^u} G\subset f (B^H\rtimes_{(\beta^H)^u} G) f.$$
Composing this map with $\Pi^u_{\beta^H,(\beta^H)^u}$, we get a $\ast$-homomorphism
	$$\widetilde{\pi}\colon A\rtimes_\alpha G \to  f (B^H\rtimes_{\beta^H} G) f$$
such that $\widetilde{\pi}(a)=\pi(a)$ for $a\in A$ and $\widetilde{\pi}( \lambda_g^\alpha ) = u_g\lambda_g^{\beta^H}$ for $g\in G$. 
The partial isometry $w$ then satisfies that, inside $M^H\rtimes_{\beta^H} G$, for all $a\in A$ and $g\in G$,
	$$\Pi_{\beta^H,\alpha^H}^{\omega^H}(a\otimes e_{1,1}) w= w\widetilde{\pi}(a) \quad \text{and} \quad \Pi_{\beta^H,\alpha^H}^{\omega^H}(\lambda_g^{\alpha^H} )w = \omega_g^H \beta_g^H (w)\lambda_g^{\beta^H} =wu_g \lambda_g^{\beta^H} =w\widetilde{\pi}( \lambda_g^\alpha ).$$
Hence using the isomorphism $M^H\rtimes_{\beta^H} G = (M\rtimes_{\beta} G) \ovt \B(H)$ and using $\Pi_{\beta^H,\alpha^H}^{\omega^H}=\Pi^{\omega}_{\beta,\alpha}\otimes \id_{H}$, 
$(H , \widetilde{\pi}, f, w)$ witnesses $\Pi^{\omega}_{\beta,\alpha}(A\rtimes_\alpha G) \preceq^{\rm uni}_{M\rtimes_{\beta} G} B\rtimes_{\beta} G$. This is equivalent to item (2) by Lemma \ref{relation for unital and corner}.

	$(1)\Rightarrow (4)$ Take $(H,\pi,f,w)$ and $(u_g)_{g\in G}$ witnessing item (1). Write $w=\sum_j w_j \otimes e_{1,j}$, where $(e_{i,j})_{i,j}$ is a matrix unit of $\B(H)$, and put $W:=\sum_j w_je_{\widetilde{B}} \otimes e_{1,j}=w e_{\widetilde{B}}^H$ (where $e_{\widetilde{B}}^H :=e_{\widetilde{B}}\otimes 1_H $). Then it satisfies that for any $a\in A$,
	$$(a\otimes e_{1,1})WW^*= (a\otimes e_{1,1})we_{\widetilde{B}}^Hw^* = w\pi(a)e_{\widetilde{B}}^Hw^* = WW^*(a\otimes e_{1,1}) ,$$
so $WW^* \in (A\otimes \C e_{1,1})' \cap (1_A\otimes e_{1,1})\langle M^H  , \widetilde{B}^H \rangle (1_A\otimes e_{1,1})= (A'\cap 1_A\langle M,\widetilde{B} \rangle1_A) \otimes \C e_{1,1}$. 
We also have that for any $g\in G$,
\begin{align*}
	\widehat{\alpha}_g^H(WW^*) 
	= \omega_g^H\widehat{\beta}_g^H(we_{\widetilde{B}}^Hw^*)(\omega_g^H)^*
	= wu_ge_{\widetilde{B}}^Hu_g^*w^* 
	=WW^*,
\end{align*}
so $WW^* \in (1_A\langle M,\widetilde{B} \rangle1_A)^{\widehat{\alpha}}\otimes \C e_{1,1}$. 
Using the equation $\widehat{E}_{\widetilde{B}\ovt \B(H)}=\widehat{E}_{\widetilde{B}}\otimes \id_H$, it holds that  
	$$(\widehat{E}_{\widetilde{B}}\otimes \id_H)(WW^*) = \widehat{E}_{\widetilde{B}\ovt \B(H)}(WW^*) = ww^*\in M\otimes \C e_{1,1}<\infty.$$ 
Thus by using the element $d$ such that $d\otimes e_{1,1} = WW^*$, we get item (4).

$(4)\Rightarrow (1)$ 
Take a nonzero spectral projection $p$ of $d$ such that $p\leq \lambda d$ for some $\lambda>0$. Then $p$ satisfies exactly the same assumption as the one of $d$. 
Fix a countably infinite dimensional Hilbert space $H$ (with a matrix unit $(e_{i,j})_{i,j}$ in $\B(H)$), and consider the inclusion 
	$$A\otimes \C e_{1,1} \subset \langle M ,\widetilde{B}\rangle \ovt \B(H) = \langle M^H ,\widetilde{B}^H \rangle .$$
Then the projection $p\otimes e_{1,1}$ satisfies that 
	$$ \widehat{E}_{\widetilde{B}^H}(p\otimes e_{1,1}) =  \widehat{E}_{\widetilde{B}}(p)\otimes e_{1,1}<\infty.$$
Since the projection $e_{\widetilde{B}}^H (1_B \otimes 1_H) = (e_{\widetilde{B}}1_B)\otimes 1_H$ is properly infinite, we can follow the proof of (6)$\Rightarrow$(2-b) of \cite[Theorem 4.3]{HI15} (we do not need the finiteness of $A$). 
We can find a partial isometry $W \in \langle M^H ,\widetilde{B}^H\rangle$ (which is of the form $w e_{\widetilde{B}}^H= W$), a projection $f\in B^H$, a $\ast$-homomorphism $\pi\colon A \to fB^H f$ such that 
$\pi(a)e^H_{\widetilde{B}}=W^*(a\otimes e_{1,1})W$ and 
$w\pi(a) = (a\otimes e_{1,1})w$ for all $a\in A$, and $WW^* = p \otimes e_{1,1} \in (1_A\langle M,\widetilde{B}\rangle1_A)^{\widehat{\alpha}} \ovt \B(H) $. 
Note that $(H,f,\pi,w)$ witnesses $A\preceq^{\rm uni}_MB$ (up to taking the polar decomposition of $w$).

	 We next construct a generalized cocycle. For any $g\in G$, since $W^*\omega^H_g \widehat{\beta}^H_g(W) \in 1_Be_{\widetilde{B}}^H \langle M,\widetilde{B}\rangle^H 1_Be_{\widetilde{B}}^H={B}^He_{\widetilde{B}}^H$, there is a unique  $u_g\in {B}^H$ such that $u_ge_{\widetilde{B}}^H=W^*\omega_g^H \widehat{\beta}_g^H(W)$. Since  $g\mapsto \omega_g^H$ and $g \mapsto \widehat{\beta}_g^H(W)$ are $\ast$-strongly continuous, so is the map $G\ni g \mapsto u_g$. Observe that 
	$$e_{\widetilde{B}}^Hu_gu_g^*=W^*\omega_g^H \widehat{\beta}_g^H(WW^*)(\omega_g^H)^*W = W^*\widehat{\alpha}_g^H(WW^*)W =fe_{\widetilde{B}}^H$$
and similarly $e_{\widetilde{B}}^Hu_g^*u_g  =
{\beta}_g^H(f)e_{\widetilde{B}}^H$ for all $g\in G$. For $g,h \in G$, we compute that 
\begin{eqnarray*}
	u_g{\beta}_g^H(u_h)e_{\widetilde{B}}^H
	&=& W^*\omega_g^H\widehat{\beta}_g^H(W)\widehat{\beta}_g^H(W^*\omega_h^H\widehat{\beta}_h^H(W)) \\
	&=& W^*\widehat{\alpha}_g^H(WW^*)\omega_g^H\widehat{\beta}_g^H(\omega_h^H)\widehat{\beta}^H_{gh}(W) \\
	&=& W^* \omega_{gh}^H\widehat{\beta}_{gh}^H(W) \\
	&=& u_{gh} e_{\widetilde{B}}^H.
\end{eqnarray*}
Thus $(u_g)_{g\in G}$ is a generalized cocycle for $\beta^H$ with support projection $f$. 
Using the equation $(\omega_g^H)^*W u_g =  \widehat{\beta}_g^H(W)$, it holds that for any $a\in A$ and $g\in G$, 
\begin{align*}
	\beta_g^H(\pi(a)) e_{\widetilde{B}}^H 
	= \widehat{\beta}_g^H(W^*(a\otimes e_{1,1})W) 
	= u_g^*W^*{\alpha}^H_g(a\otimes e_{1,1}) W u_g
	= u_g^*\pi({\alpha}_g(a))u_g e_{\widetilde{B}}^H. 
\end{align*}
We get the equivariance property $ u_g\beta_g^H(\pi(a))u_g^* = \pi(\widehat{\alpha}_g(a))$  for all $a\in A$. 
Finally, since $W=w e_{\widetilde{B}}^H$, the equation $(\omega_g^H)^*W u_g = \widehat{\beta}_g^H(W)$ for $g\in G$ implies $(\omega_g^H)^*w  u_g e_{\widetilde{B}}^H= \beta_g^H(w) e_{\widetilde{B}}^H$. 
We get $w  u_g = \omega_g^H\beta_g^H(w)$ for all $g\in G$, and thus $(u_g)_{g\in G}$ is a desired cocycle. We get item (1).

\bigskip

\noindent
{\bf From now on, we assume that $A\rtimes_\alpha G$ is finite.}

	$(2) \Leftrightarrow (3)$ 
	Assume $A\rtimes_\alpha G$ is finite. Suppose first that item (3) does not hold, hence there exists a net $(u_i)_{i}$ of unitaries in $\mathcal U(A)$ and $(g_i)_{i}$ in $G$ such that 
	$$E_B(\beta_{g_i}(b^*)\omega_{g_i}^* u_i a ) \rightarrow 0, \quad \text{$\sigma$-strongly for all $a,b\in M1_B$.}$$ 
Then for any $a,b\in M1_B$ and $s,s'\in G$, we have 
\begin{align*}
	& E_{B\rtimes_\beta G}(\lambda_s^\beta b^*  \Pi^\omega_{\beta,\alpha}(\lambda_{g_i^{-1}}^\alpha)u_i a\lambda_{s'}^\beta) \\
	=& \lambda_s^\beta E_{B\rtimes_\beta G}(b^*  \lambda_{g_i^{-1}}^\beta \omega_{g_i}^* u_i a)\lambda_{s'}^\beta  \\
	=& \lambda_{sg_i^{-1}}^\beta E_{B}(\beta_{g_i}(b^*  )\omega_{g_i}^* u_i a)\lambda_{s'}^\beta.
\end{align*}
The last term converges to 0 in the $\sigma$-strong topology for all $a,b\in M1_B$ and $s,s' \in G$. By Theorem \ref{thm corner intertwining}(3) (see also \cite[Theorem 4.3(5)]{HI15}), this means $\Pi^{\omega}_{\beta,\alpha}(A\rtimes_\alpha G) \not\preceq_{M\rtimes_\beta G} B\rtimes_\beta G$. 

	Conversely Suppose that $\Pi^{\omega}_{\beta,\alpha}(A\rtimes_\alpha G) \not\preceq_{M\rtimes_\beta G} B\rtimes_\beta G$. Then by Theorem \ref{thm corner intertwining}(3), there exist a net $(u_i)_{i}$ of unitaries in $\mathcal U(A)$ and $(g_i)_{i}$ in $G$ such that 
	$$E_{B\rtimes_\beta G}(y^*  \Pi^\omega_{\beta,\alpha}(\lambda_{g_i^{-1}}^\alpha)u_i x) \rightarrow 0, \quad \text{$\sigma$-strongly for all $x,y\in (M\rtimes_\beta G) 1_B$.}$$
Using the same computation as above, we get that item (3) does not hold.

$(3)\Rightarrow (4)$ 
Assume that $A\rtimes_\alpha G$ is finite. 
Let $\psi$ be a faithful normal state on $M\rtimes_\alpha G$ which is preserved by $E_{A\rtimes_\alpha G}$ such that $\psi|_{A\rtimes_\alpha G}$ is a trace. 
Observe that $\psi|_{1_A M 1_A}$ is $\alpha$-preserving, since $1_A \lambda_g^\alpha \in (1_A M 1_A)_\psi$ for all $g\in G$.
It then holds that $\widehat{\psi}\circ \widehat{\alpha}_g = \widehat{\psi}$ on $(1_A \langle M,\widetilde{B} \rangle1_A J1_BJ)^+$ for all $g\in G$. 

By assumption, there exist $\delta>0$ and a finite subset $\mathcal{F}\subset 1_AM1_B$ such that 
	$$\sum_{a,b\in \mathcal{F}}\|E_{B}(\beta_g(b^*) w_{g}^* u a)\|_{2,\varphi}^2>\delta, \quad \text{for all } u\in\mathcal{U}(A), \  g\in G.$$ 
Put $d_0:=\sum_{y\in \mathcal{F}}ye_{{\widetilde{B}}}y^*\in (1_A\langle M,\widetilde{B} \rangle1_A)^+ $ and observe that $d_0=d_0 J1_BJ$, $\widehat{E}_{\widetilde{B}}(d_0) = \sum_{y \in \mathcal F} yy^* \in 1_AM1_A$ and $\mathrm{ctr}(d_0 \, J1_{B_1}J) = \sum_{y\in \mathcal F} J\mathrm{ctr}_{B_1}(E_B(1_{B_1}y^*y1_{B_1})) J< +\infty$. 
Define 
	$$\mathcal{K}:=\overline{\mathrm{co}}^{\rm weak}\left \{ u^* \widehat{\alpha}_g(d_0) u \mid u \in \mathcal U(A),\ g\in G \right \}\subset 1_A\langle M,\widetilde{B} \rangle 1_A.$$
Following the proof of (5)$\Rightarrow$(6) of \cite[Theorem 4.3]{HI15}, there exists a unique element $d \in \mathcal K$ of minimum $\|\cdot\|_{2,\widehat{\psi}}$-norm. Since $\widehat{\psi}$ is preserved by $\widehat{\alpha}$ and since $A$ is contained in the centralizer of $\widehat{\psi}$, we get that $d\in A'\cap (1_A\langle M,\widetilde{B}\rangle 1_A)^{\widehat{\alpha}}$. Note that $d = d J 1_{B} J$, since $d_0 = d_0 J1_BJ$.

	We prove that $d\neq0$. For all $u \in \mathcal{U}(A)$ and $g\in G$, we have
\begin{align*}
	\sum_{a\in\mathcal{F}}\langle u^*\widehat{\alpha}_{g}(d_0)u \, \Lambda_\varphi(a),\Lambda_\varphi(a)\rangle_{\varphi}
	=&\sum_{a,b\in\mathcal{F}}\langle    u^*  \widehat{\alpha}_g(b e_{\widetilde{B}} b^*)u  \Lambda_\varphi(a), \Lambda_\varphi(a)\rangle_{\varphi}\\
	=&\sum_{a,b\in\mathcal{F}}\langle    u^*  w_{g}\beta_{g}(b) e_{B} \beta_{g}(b^*)w_{g}^*u  \Lambda_\varphi(a), \Lambda_\varphi(a)\rangle_{\varphi}\\
	=&\sum_{a,b\in \mathcal{F}}\|E_{B}(\beta_g(b^*) w_{g}^* u a)\|^2_{2,\varphi_B}>\delta .
\end{align*}
By taking convex combinations and a $\sigma$-weak limit, we obtain $\sum_{a\in\mathcal{F}}\langle d  \Lambda_\varphi(a),\Lambda_\varphi(a)\rangle_{\varphi} \geq \delta$. This implies $d \neq 0$.

We prove $\widehat{E}_{\widetilde{B}}(d)\in M$. 
Observe that for any $g\in G$,
\begin{align*}
	\widehat{E}_{\widetilde{B}}(u^*\widehat{\alpha}_g(d_0)u) 
	&=\sum_{y\in \mathcal{F}}\widehat{E}_{\widetilde{B}}(u^*{\alpha}_g(y) \omega_ge_{\widetilde{B}}\omega_g^* {\alpha}_g(y^*)u )\\
	&=\sum_{y\in \mathcal{F}}u^*{\alpha}_g(y)  {\alpha}_g(y^*)u \\
	&=u^*{\alpha}_g\left(\sum_{y\in \mathcal{F}}yy^*  \right)u. 
\end{align*}
Combined with the normality of $\widehat{E}_{\widetilde{B}}$, we conclude that $\|\widehat{E}_{\widetilde{B}}(x)\|_\infty \leq \|\sum_{y\in \mathcal{F}}yy^*\|_\infty$ for all $x \in \mathcal{K}$, hence $\widehat{E}_{\widetilde{B}}(d)\in M$. We get item (4).

	Finally we prove that the Hilbert space $H$ in item (1) can be taken as finite dimensional. For this, we continue to use $d_0,d,\mathcal K$ and claim $\mathrm{ctr}(d J1_{B_1}J)<\infty$. Using the formula for $\mathrm{ctr}$ given in Section \ref{Preliminaries} and using $\mathrm{ctr}_{B_1}\circ \beta_g = \beta_g \circ \mathrm{ctr}_{B_1}$ on $B_1$ for all $g\in G$, we compute that for any $g\in G$ and $u\in \mathcal{U}(A)$
\begin{align*}
	\mathrm{ctr}(u^*\widehat{\alpha}_g(d_0)u \, J1_{B_1}J) 
	&= \sum_{y\in \mathcal{F}}\mathrm{ctr}([u^*\omega_g{\beta}_g(y)] e_{\widetilde{B}} [{\beta}_g(y^*)\omega_g^*u] \, J1_{B_1}J)\\
	&= \sum_{y\in \mathcal{F}}J\mathrm{ctr}_{B_1}\circ E_B(1_{B_1}[{\beta}_g(y^*)\omega_g^*u][{\beta}_g(y^*)\omega_g^*u]^*1_{B_1})J\\
	&= \sum_{y\in \mathcal{F}}J\mathrm{ctr}_{B_1}\circ E_B(1_{B_1}{\beta}_g(y^*y)1_{B_1})J\\
	&= J{\beta}_g\circ\mathrm{ctr}_{B_1}\circ E_B(\sum_{y\in \mathcal{F}} 1_{B_1}y^*y1_{B_1})J.
\end{align*}
Combined with the normality of $\mathrm{ctr}$, we get 
	$$\|\mathrm{ctr}(x J1_{B_1}J)\|_\infty \leq \|\mathrm{ctr}_{B_1}(E_B(\sum_{y\in \mathcal{F}}1_{B_1}y^*y 1_{B_1})\|_\infty$$ 
for all $x\in \mathcal{K}$. Thus we get $\mathrm{ctr}(d J1_{B_1}J)<\infty$.

	We next follow the proof of (4)$\Rightarrow$(1) above. Take a nonzero spectral projection $p$ of $d$ such that $p\leq \lambda d$ for some $\lambda>0$, so that $\mathrm{ctr}(d J1_{B_1}J)<\infty$ and $\widehat{E}_{\widetilde{B}}(p)\in M$. We have either $p \, J1_{B_1}J\neq0$ or $p \, J1_{B_2}J\neq0$. 

	Assume that $p \, J1_{B_2}J\neq0$. We may assume $p \, J1_{B_2}J = p$. Then since $B_2$ is properly infinite, we can follow the proof above (with $H=\C$ and $B=B_2$), so we get item (1) with $H=\C$. 

	Assume that $p \, J1_{B_1}J\neq0$ and we may assume $p \, J1_{B_1}J = p$. Then using $\widehat{E}_{\widetilde{B}}(p)<\infty$ and $\mathrm{ctr}(p)<\infty$, there is a family $\{w_i\}_{i=1}^n \subset M1_{B_1}$ such that $W_i:=w_ie_{\widetilde{B}}$ are partial isometries for all $i$, $p = \sum_{i=1}^n w_i e_{\widetilde{B}} w_i^* = \sum_{i=1}^n W_iW_i^*$, and $E_B(w_i^*w_j) = \delta_{i,j}p_j$ for all $i,j$, where $p_j\in B_1$ are projections. 
(Indeed using $\widehat{E}_{\widetilde{B}}(p)<\infty$, one can first choose $\{p_i\}_{i \in I}$ as above but possibly $|I|=\infty$. Using a maximality argument, we can assume that the central support of $p_{i+1}$ in $B_1$ is smaller than $p_{i}$ for all $i$. Then using $\mathrm{ctr}(p)<\infty$, the family $\{p_i\}_{i}$ must be a finite set.) 
Consider a $\ast$-homomorphism $\pi\colon p\langle M,\widetilde{B} \rangle p \to B_1 \ovt \M_n$ given by
\begin{align*}
	pxp=\sum_{i,j=1}^nW_i(W_i^* x W_j)W_j^* \mapsto \sum_{i,j=1}^n E_B(w_i^*xw_j) \otimes e_{i,j}, \quad (x\in \langle M,\widetilde{B}\rangle).
\end{align*}
Then using the identification $p\langle M,\widetilde{B} \rangle p \simeq p\langle M,\widetilde{B} \rangle p \otimes \C e_{1,1}$ and the partial isometry $W:=\sum_{j} W_j \otimes e_{1,j}$, the map $\pi$ satisfies $	\pi(x)(e_{\widetilde{B}} \otimes 1_n)=W^* (x\otimes e_{1,1})W$ for all $x\in p\langle M,\widetilde{B} \rangle p$. 
Define $f:=\pi(1_A)\in B_1\otimes \M_n$  and $w:=\sum_{j} w_j \otimes e_{1,j} \in M \otimes \M_n$, so that $W^*W=f(e_{\widetilde{B}}\otimes 1_n)$ and $W=w(e_{\widetilde{B}}\otimes 1_n)$. 
By restricting $\pi$ to $Ap$ and composing with the map $A \to Ap$, we have a unital normal $\ast$-homomorphism $\pi \colon A \to f(B_1\otimes \M_n)f$ such that $ (a\otimes e_{1,1})W = W\pi(a)$ for all $a\in A$. 
Thus we are exactly in the same situation as in the proof of (4)$\Rightarrow$(1) but with $H=\C^n$ and $B=B_1$. Following the same proof, we get item (1) with $H=\C^n$ as desired. 
\end{proof}

\subsection*{Intertwining theory with modular actions}

	We next focus on the case of modular actions. We continue to use $A,B \subset M$ and fix faithful normal conditional expectations $E_A,E_B$ for $A,B$ respectively. 
	Let $\psi,\varphi \in M_*$ be faithful normal positive functionals which are preserved by $E_A,E_{{B}}$ respectively. 
Then since $\sigma^\psi_t(A)=A$, $\sigma_t^\varphi(B)=B$ for all $t\in \R$, and $\sigma^\psi$ and $\sigma^\varphi$ are cocycle conjugate by $([D\psi:D\varphi]_t)_{t\in \R}$, one can think the condition $(A,\sigma^\psi) \preceq^{\rm uni}_M (B,\sigma^\varphi)$. 
In this setting, the extended actions of $\sigma^\psi$ and $\sigma^\varphi$ on $\langle M,\widetilde{B}\rangle$ are exactly the modular actions of $\widehat{\psi}:=\psi\circ \widehat{E}_{\widetilde{B}}$ and $\widehat{\varphi}:=\varphi\circ \widehat{E}_{\widetilde{B}}$ respectively. 

	As in the usual intertwining condition, we introduce intertwining conditions with modular actions at a level of \textit{corners}.
\begin{Def}\label{def corner intertwining with modular actions}\upshape
	Keep the setting. We will say that \textit{a corner of $(A,\sigma^\psi)$ embeds with expectation into $(B,\sigma^\varphi)$ inside $M$} and write $(A,\sigma^\psi) \preceq_M (B,\sigma^\varphi)$ if there exist $(e,f,\theta,v)$ which witnesses $A\preceq_M  B$ with $e\in A_\psi$, and a generalized cocycle $(u_t)_{t\in \R}$ for $\sigma^\varphi$ with values in $B$ and with support projection $f$ such that, with $\omega_t:=[D\psi:D\varphi]_t$,
		\begin{itemize}
			\item $vu_t =  \omega_t \sigma^\varphi_t (v)$ for all $t\in \R$;
			\item $u_t\sigma^\varphi_t(\theta(a))u_t^* = \theta(\sigma^\psi_t(a))$, for all $a\in eAe$ and $t\in \R$.
		\end{itemize}
In this case, we will say that $(e,f,\theta,u)$ and $(u_g)_{g\in G}$ \textit{witness $(A,\sigma^\psi) \preceq_M (B,\sigma^\varphi)$}.
\end{Def}

	Below we collect elementary lemmas. We omit proofs since they are straightforward. 

\begin{Lem}\label{lemma for equivalence}
	Assume $(A,\sigma^\psi)\preceq_M (B,\sigma^\varphi)$ and fix $(e,f,\theta,v)$ and $(u_t)_{t\in \R}$ which witness $(A,\sigma^\psi)\preceq_M (B,\sigma^\varphi)$ as in the sense of Definition \ref{def corner intertwining with modular actions}. 
\begin{itemize}
	\item[$\rm (1)$] For any projection $e_0 \in eA_\psi e$ with $e_0v=v\theta(e_0)\neq 0$, $(e_0,\theta(e_0),\theta|_{e_0Ae_0}, e_0v)$ and $(\theta(e_0)u_t )_{t\in \R}$ witness $(A,\sigma^\psi)\preceq_M (B,\sigma^\varphi)$ (up to the polar decomposition of $e_0v$). 

	\item[$\rm (2)$] For any projection $z \in  B \cap \theta(eAe)' \cap \{u_t\mid t\in \R\}'$ (e.g.\ $z\in \mathcal{Z}(B)$) with $vz\neq 0$, $(e,f z, \theta(\, \cdot\, ) z, vz)$ and $(u_t z)_{t\in \R}$ witness $(A,\sigma^\psi)\preceq_M (B,\sigma^\varphi)$ (up to the polar decomposition of $vz$). 

	\item[$\rm (3)$] Let $u\in A$ and $w\in B$ be partial isometries such that $e=u^*u$ and $f=ww^*$. 
Then $(uu^*,w^*w,\Ad(w^*)\circ \theta \circ \Ad(u^*),uvw)$ and the generalized cocycle $(w^*u_t \sigma^\varphi_t(w))_{t\in \R}$ witness $(A,\sigma^{\psi'})\preceq_M (B,\sigma^{\varphi})$, where $\psi'\in M_*^+$ is any faithful element which is preserved by $E_A$  such that $uu^* \psi' uu^* = u\psi u^*$ and $uu^* \in A_{\psi'}$. 

	\item[$\rm (4)$] Let $\psi'$ and $\varphi'$ be any faithful normal positive functionals on $M$  which are preserved by $E_A$ and $E_B$ respectively such that $e\in A_{\psi'}$. 
Then $(e,f,\theta,v)$ and $(\theta(e[D\psi':D\psi]_te)u_t [D\varphi:D\varphi']_t)_{t}$ witness $(A,\sigma^{\psi'}) \preceq_M (B,\sigma^{\varphi'})$.

\end{itemize}
Moreover all these statements hold if we consider $(H,f,\pi,w)$ and $(u_t)_{t\in \R}$ which witness $(A,\sigma^\psi)\preceq_M^{\rm uni} (B,\sigma^\varphi)$ as in the sense of Definition \ref{def unital intertwining with actions}. 
(In this case, we use $\mathcal{Z}(A)$ and $B\ovt \B(H)$, instead of $A_\psi $ and $B$ in items $(1)$,$(2)$, and $(3)$, and item $(4)$ holds without the assumption $e\in A_{\psi'}$). 
\end{Lem}

	The next lemma clarifies the relation between $\preceq$ and $\preceq^{\rm uni}$ for modular actions. It should  be compared to Lemma \ref{relation for unital and corner}.

\begin{Lem}\label{relation for unital and corner of actions}
	The following assertions hold true.
\begin{itemize}
	\item[$\rm (1)$] We have that $(A,\sigma^\psi) \preceq_M (B,\sigma^\varphi)$ holds if and only if $(A,\sigma^\psi) \preceq^{\rm uni}_M (B,\sigma^\varphi)$ holds. In particular, these notions do not depend on the choice of $\psi$ and $\varphi$ (as long as they are preserved by $E_A$ and $E_B$ respectively).

	\item[$\rm (2)$]
	Assume either one of the following conditions holds: 
\begin{itemize}
	\item $A$ does not have any direct summand which is semifinite and properly infinite; or
	\item $B$ is properly infinite.  
\end{itemize}
	If $(A,\sigma^\psi) \preceq^{\rm uni}_M (B,\sigma^\varphi)$ holds, then the Hilbert space $H$ in Definition \ref{def unital intertwining with actions} can be taken as finite dimensional.
\end{itemize}
\end{Lem}
\begin{proof}
	We decompose $A=A_1\oplus A_2\oplus A_3$ and $B=B_1\oplus B_2\oplus B_3$, where $A_1,B_1$ are finite, $A_2,B_2$ are semifinite and properly infinite, and $A_3,B_3$ are of type III. 
Then by Lemma \ref{lemma for equivalence}(1),(2) and \cite[Remark 4.2(2)]{HI15}, we have that $(A,\sigma^\psi) \preceq_M (B,\sigma^\varphi)$ holds if and only if $(A_i,\sigma^\psi) \preceq_M (B_j,\sigma^\varphi)$ holds for some $i,j$. 
Hence we can always assume that $A=A_i$ and $B=B_j$ for some $i,j$. The same thing is true for $(A,\sigma^\psi) \preceq^{\rm uni}_M (B,\sigma^\varphi)$.

	(1) By Lemma \ref{lemma for equivalence}(4), the condition $(A,\sigma^{\psi}) \preceq^{\rm uni}_M (B,\sigma^\varphi)$ does not depend on the choice of $\psi,\varphi$. Hence if this statement is proven, then $(A,\sigma^{\psi}) \preceq_M (B,\sigma^\varphi)$ also does not depend on $\psi,\varphi$.

	Assume that $(A_i,\sigma^\psi) \preceq^{\rm uni}_M (B_j,\sigma^\varphi)$ holds for some $i,j$ and take $(H,f,\pi,w)$ and $(u_t)_t$ as in the definition. 
Let $z\in \mathcal{Z}(A)$ be a nonzero projection such that $Az \ni a \mapsto \pi(a)w^*w$ is injective.  Since $z\in A_\psi$, up to exchanging $Az$ by $A$, we may assume that $A\ni a \mapsto \pi(a)w^*w$ is injective. In particular $w\pi(e) \neq 0$ for any nonzero projection $e\in A$. 

	Assume that $B=B_2$ or $B=B_3$. Then since $1_B \otimes e_{1,1}$ is properly infinite, one has $f \prec 1_B \otimes e_{1,1}$. Up to equivalence of projections, using Lemma \ref{lemma for equivalence}(3), we may assume that $f$ is contained in $B\otimes \C e_{1,1}$. So using $M=M\otimes \C e_{1,1}$, we get $(A,\sigma^\psi) \preceq_M (B,\sigma^{\varphi})$.

	Assume that $B=B_1$. Then we must have that $A=A_1$ or $A_2$. If $A=A_2$, then by using $eA e$ for any fixed finite projection $e\in A_\psi$ (note that $A_\psi$ contains many finite projections, e.g.\ the first part of the proof of \cite[Lemma 2.1]{HU15}) and using Lemma \ref{lemma for equivalence}(1), we may assume that $A$ is finite. 
By the last statement of Theorem \ref{thm unital intertwining with actions}, we may assume that $A$ is finite and $H$ is finite dimensional. We can still assume that $A\ni a \mapsto \pi(a)w^*w$ is injective.

	Write $H=\C^n$ for some $n\in \N$. As in the proof of \cite[Proposition F.10]{BO08} or \cite[Proposition 3.1 (ii)$\Rightarrow$(iii)]{Ue12}, there is a projection $e\in A$ such that $\pi(e)$ is equivalent to a projection $f_0\otimes e_{1,1}$ for some $f_0 \in B$. By \cite[Lemma 2.1]{HU15}, $e$ is equivalent to a projection in $A_\psi$, so we may assume $e\in A_\psi$. 
Observe that, regarding $\pi$ as a map from $A\otimes \C e_{1,1}$, $(1_A\otimes e_{1,1},f,\pi,w)$ and $(u_t)_t$ witness $(A\otimes \C e_{1,1}, \sigma^\psi)\preceq_{M\otimes \M_n} (B\otimes \M_n, \sigma^{\varphi\otimes \mathrm{tr}_n})$. 
Since $\pi(e)w^*w\neq 0$, by Lemma \ref{lemma for equivalence}(1), $(e\otimes e_{1,1}, \pi(e), \pi|_{eAe\otimes e_{1,1}}, (e\otimes e_{1,1})w)$ witness $(A\otimes \C e_{1,1}, \sigma^\psi)\preceq_{M\otimes \M_n} (B\otimes \M_n, \sigma^{\varphi\otimes \mathrm{tr}_n})$ as well. 
We then apply Lemma \ref{lemma for equivalence}(3) for $\pi(e) \sim f_0\otimes e_{1,1}$, and obtain that $(e\otimes e_{1,1}, f_0\otimes e_{1,1}, \pi', w')$ and some generalized cocycle witness $(A\otimes \C e_{1,1}, \sigma^\psi)\preceq_{M\otimes \M_n} (B\otimes \M_n, \sigma^{\varphi\otimes \mathrm{tr}_n})$ for some $\pi'$ and $w'$. 
Finally since $f_0\otimes e_{1,1}$ and $w'$ are contained in $M\otimes \C e_{1,1}$, by identifying $M\otimes \C e_{1,1} = M$, we get $(A, \sigma^\psi)\preceq_{M} (B, \sigma^{\varphi})$.

	We next show the `only if' direction. Assume that $(A,\sigma^\psi) \preceq_M (B,\sigma^\varphi)$ holds and take $(e,f,\theta,v)$ and $(u_t)_t$ as in the definition. As in the proof above, we can assume $eAe \ni a \mapsto v^*v\theta(a)$ is injective and hence $v\theta(e_0) \neq 0$ for any nonzero projection $e_0\in eAe$. 

	Let $z$ be the central support projection of $e$ in $A$, and take  partial isometries $(w_i)_{i\in I}$ in $A$ such that $w_0=e$, $e_i:=w_i^*w_i\leq e$ for all $i \in I$, and $\sum_{i \in I} w_iw_i^* =z$. Note that $I$ is a countable set, so we regard $I\subset \N$. 
We put $v_n:=w_n v $ for all $n \in I$ and $d=\sum_{n\in I} v_n e_{\widetilde{B}} v_n^*$, and then it is easy to see that 
$d=dJ1_BJ$ and $\widehat{E}_{\widetilde{B}}(d)\in M .$  
We note that $d\neq 0$, since each $v_n$ is nonzero by $w_n^*v_n = w_n^*w_n v=v \theta(w_n^*w_n) \neq 0$. 
It is easy to compute that $ad=da$ for all $a\in A$, hence $d \in A' \cap 1_A \langle M,\widetilde{B}\rangle1_A$. 
Define a faithful normal positive functional $\psi'$ on $M$ by 
	$$ \psi' := \sum_{n \in I} \frac{1}{2^{n}} \, w_n  \psi   w_n^* + (1-z)\psi (1-z).$$
Note that $\psi'$ is preserved by  $E_A$. By Lemma \ref{connes cocycle lemma}, the equation $e_n\psi' e_n = 2^{-n} w_n\psi w_n^*$ implies $ \sigma_t^{\psi}(w_n) = 2^{-itn} [D\psi':D\psi]_t^* w_n$ for all $t\in \R$ and $n\in I$. 
An easy computation shows that 
\begin{align*}
	\sigma_t^{\widehat{\psi}}(d)
	= [D\psi:D\varphi]_t\sigma_t^{\widehat\varphi}(d)[D\psi:D\varphi]_t^* 
	= [D\psi':D\psi]_t^* \, d\, [D\psi':D\psi]_t, \quad \text{for all }t\in \R.
\end{align*}
We get that $\sigma_t^{\widehat{\psi}'}(d) = d$ for all $t\in \R$ and hence $d\in A' \cap (1_A \langle M,\widetilde{B}\rangle1_A)_{\widehat{\psi}'}$. 
By Theorem \ref{thm unital intertwining with actions}, this means $(A,\sigma^{\psi'}) \preceq^{\rm uni}_M (B,\sigma^\varphi)$. By Lemma \ref{lemma for equivalence}(4), this is equivalent to $(A,\sigma^{\psi}) \preceq^{\rm uni}_M (B,\sigma^\varphi)$.

	(2) Assume that $(A_i,\sigma^\psi) \preceq^{\rm uni}_M (B_j,\sigma^\varphi)$ holds for some $i,j$. If $B=B_2$ or $B_3$, then the first half of the proof of item (1) shows that one can assume $H=\C$. So we get the conclusion. 
If $A=A_3$, then we must have $B=B_3$, which we proved. 
Finally if $A=A_1$, then the last part of Theorem \ref{thm unital intertwining with actions} gives the conclusion.
\end{proof}

\subsection*{Intertwining theory with conditional expectations}

	In \cite{HSV16}, a notion of intertwining conditions for \textit{states} was introduced. Inspired from this, we introduce a notion of intertwining conditions for \textit{conditional expectations}. 
We still fix $A,B \subset M$ with expectations $E_A,E_B$. 

\begin{Def}\label{def corner intertwining with expectation}\upshape
	We say that \textit{a corner of $(A,E_A)$ embeds with expectation into $(B,E_B)$  inside $M$} and write $(A,E_A) \preceq_M (B,E_B)$ if there exist $(e,f,\theta,v)$ which witnesses $A\preceq_M B$, and faithful normal positive functionals $\psi,\varphi \in M_*$ which are preserved by $E_A,E_B$ respectively such that 
	$$vv^* \in (1_AM1_A)_\psi, \quad v^*v\in (1_BM1_B)_{\varphi}, \quad \text{and} \quad vv^* \psi vv^* = v\varphi v^*. $$
In this case, we say that $(e,f,\theta,v)$ and $\psi,\varphi$ \textit{witness} $(A,E_A) \preceq_M (B,E_B)$.
\end{Def}

The next lemma clarifies relations between $A\preceq_MB$ and $(A,E_{A})\preceq_M(B,E_B)$. Note that, as in the statement of Theorem \ref{thmA}, one can actually take $q=1_A$ in the next lemma (which will be proved later). 

\begin{Lem}\label{intertwining implies expectation}
	The condition $A\preceq_MB$ holds if and only if there is a nonzero projection $q\in A'\cap 1_AM1_A$ and a faithful normal conditional expectation $E_{Aq}\colon qMq \to Aq$ such that $(Aq,E_{Aq})\preceq_M(B,E_B)$.
\end{Lem}
\begin{proof}
	The `if' direction is trivial, so we see the `only if' direction. 
Take $(e,f,\theta,v)$ which witnesses the condition $A\preceq_M B$. 
By \cite[Remark 4.2(2),(3)]{HI15}, we may assume that $A$ is finite or of type III, and that $eAe\ni a \mapsto \theta(a)v^*v$ is injective. 
Up to exchanging $e$ with a small one if necessary, we may assume that there exist finitely many orthogonal and equivalent projections $(e_{i})_{i=1}^n$ in $A$ such that $\sum_{i=1}^ne_i =:z_A(e)\in \mathcal{Z}(A)$. 
Fix a faithful normal conditional expectation $E_\theta$ for the inclusion $\theta(eAe) \subset fBf$, and take a faithful normal state $\varphi_B$ on $B$ such that $\varphi_B \circ E_{\theta} = \varphi_B$ on $fBf$. Put $\varphi:= \varphi_B \circ E_B$ on $1_BM1_B$ and observe that the modular action of $\varphi$ globally preserves $\theta(eAe)$ and $fBf$. 
In particular it also preserves $\theta(eAe)'\cap fMf$, so using \cite[Lemma 2.1]{HU15}, there is a partial isometry $w\in \theta(eAe)'\cap fMf$ such that $w^*w = v^*v$ and $ww^* \in (\theta(eAe)'\cap fMf)^{\sigma^{\varphi}}$. Up to exchanging $vw^*$ by $v$, we may assume that $v^*v$ is contained in $(fMf)^{\sigma^{\varphi}}$. 

	We put $e_0:=vv^*\in (eAe)' \cap eMe$ and $f_0:=v^*v \in (\theta(eAe)'\cap fMf)^{\sigma^{\varphi}}$. Since $ \theta(eAe) f_0 \subset f_0Mf_0$ is globally preserved by $\sigma^\varphi$, it is with expectation, say $E\colon f_0Mf_0 \to \theta(eAe) f_0 $, which satisfies $\varphi\circ E = \varphi$ on $f_0Mf_0$. Observe that $\Ad(v)$ gives a spacial isomorphism  from $\theta(eAe) f_0$ onto $(eAe)e_0$. Hence we can define a conditional expectation by
	$$E_A':=\Ad(v)\circ E \circ \Ad(v^*)\colon e_0Me_0 \to (eAe)e_0.$$
Define a positive functional $\psi_A':=v \varphi v^*$ on $(eAe)e_0$ and put $\psi':=\psi_A'\circ E_A'$ on $e_0Me_0$. 
It holds that $v^*v =f_0\in (1_BM1_B)_\varphi$ and $vv^* = e_0 \in (e_0Me_0)_{\psi'}$. By using $\psi_A'=v \varphi v^*$ on $(eAe)e_0$ and $\varphi\circ E=\varphi$ on $f_0Mf_0$, we compute that, for any $x\in M$
\begin{align*}
	vv^*\psi'(x)vv^*
	&= \psi_A'\circ E_A'(vv^*xvv^*) \\
	&= (v\varphi v^*)( vE(v^* vv^*xvv^* v) v^*)\\
	&= \varphi ( f_0E(v^*xv) f_0)\\
	&= \varphi \circ E(v^*xv)\\
	&= \varphi (v^*xv).
\end{align*}
We get $vv^*\psi' vv^*=v\varphi v^*$. Since they satisfy $\varphi=\varphi\circ E_B$ on $1_BM1_B$ and $\psi'=\psi' \circ E_A'$ on $e_0Me_0$, we can extend $\varphi$ and $\psi'$ to ones on $M$ which are preserved by $E_B$ and $E_A'$ respectively. In this case, we still have that $f_0 \in M_\varphi$, $e_0 \in M_{\psi'}$, and $vv^*\psi' vv^*=v\varphi v^*$.

	We claim $((eAe)e_0, E_A') \preceq_M (B,E_B)$. Let $z\in \mathcal{Z}(eAe)$ be the central support projection of $e_0$ in $(eAe)'$ and observe that $(eAe)e_0 \simeq eAe z$. 
Since we assumed $eAe\ni a \mapsto v^*v\theta(a)=v^* a v$ is injective, the map $eAe\ni a \mapsto \Ad(v)(v^*v\theta(a))=ae_0$ is also injective. In particular we get $z=e$ and $(eAe)e_0 \simeq eAe$. 
Consider $\theta_0\colon (eAe)e_0\simeq eAe \to^{\theta} fBf$ given by $\theta_0(a e_0) := \theta(a)$ for $a\in eAe$. 
Then $(ee_0,f,\theta_0,v)$ witnesses $(eAe)e_0\preceq_MB$. Combined with  $\varphi$ and $\psi'$ together, we obtain $((eAe)e_0, E_A') \preceq_M (B,E_B)$.

	Since $e_0 \in (eAe)' \cap (eMe) = (A'\cap 1_AM1_A)e$, there is a projection $q \in A'\cap 1_AM1_A$ such that $qe = e_0$ and $q=z_A(e)q$. 
Using projections $(e_i)_{i=1}^n$ which we fixed at the first paragraph, we have an identification $qMq \simeq e_0Me_0\otimes \M_n$ which restricts $Aq \simeq eAeq\otimes \M_n$. In particular, there is a faithful normal conditional expectation $E_{Aq}\colon qMq \to Aq$ such that $E_{Aq}|_{e_0Me_0}=E_A'$. 
Since we chose $\psi'$ as any extension of $\psi'|_{e_0Me_0}$ which is preserved by $E_A'$, we can particularly choose $\psi'$ as the one which is preserved by $E_A'$ and $E_{Aq}$. 
Then it is easy to see that the same $(ee_0,f,\theta_0,v)$ as above and $\psi',\varphi$ witness $(Aq, E_{Aq}) \preceq_M (B,E_B)$. 
\end{proof}

	The next theorem clarifies the relation between $(A,E_A)\preceq_M(B,E_B)$ and $(A,\sigma^{\psi})\preceq_M(B,\sigma^{\varphi})$. The proof uses Connes cocycles to construct a positive functional. Note that the case $A=\C$ was proved in (the proof of) \cite[Theorem 3.1]{HSV16}.

\begin{Thm}\label{relation for expectations and actions}
	We have that $(A,E_A)\preceq_M(B,E_B)$ if and only if there exist faithful normal states $\psi,\varphi\in M_*$ which are preserved by $E_A,E_B$ respectively such that $(A,\sigma^{\psi})\preceq_M(B,\sigma^{\varphi})$.
\end{Thm}
\begin{Rem}\upshape
	Combined with Lemma \ref{relation for unital and corner of actions}(1), characterizations given in Theorem \ref{thm unital intertwining with actions} can be adapted to $(A,E_A)\preceq_M(B,E_B)$ and $(A,\sigma^{\psi})\preceq_M(B,\sigma^{\varphi})$. Moreover $\psi$ and $\varphi$ for $(A,\sigma^{\psi})\preceq_M(B,\sigma^{\varphi})$ can be taken arbitrary as long as they are preserved by $E_A$ and $E_B$ respectively.
\end{Rem}
\begin{proof}
	Suppose $(A,E_A)\preceq_M(B,E_B)$ and take $(e,f,\theta,v)$ and $\psi,\varphi$. We put $d:=ve_{\widetilde{B}}v^*$ and observe that  $d\in (eAe)'\cap (e\langle M,\widetilde{B}\rangle e)$, $d=dJ1_BJ$, and $\widehat{E}_{\widetilde{B}}(d)<\infty$. 
By Lemma \ref{connes cocycle lemma}, the equation $vv^* \psi vv^* = v\varphi v^*$ implies $[D\psi : D\varphi]_t\sigma_t^\varphi(v) =  v $ for all $t\in \R$. 
It then holds that $\sigma_t^{\widehat{\psi}}(d)= d$ for any $t\in \R$, hence $d\in A' \cap (1_A \langle M,\widetilde{B}\rangle1_A)_{\widehat{\psi}}$. 
We get that $(eAe,\sigma^\psi)\preceq^{\rm uni}_M (B,\sigma^\varphi)$ by Theorem \ref{thm unital intertwining with actions}.
This implies $(eAe,\sigma^\psi)\preceq_M (B,\sigma^\varphi)$ by Lemma \ref{relation for unital and corner of actions}, and hence $(A,\sigma^\psi)\preceq_M (B,\sigma^\varphi)$. 

	Suppose $(A,\sigma^{\psi})\preceq_M(B,\sigma^{\varphi})$ and take $(e,f,\theta,v)$ and $(u_t)_{t\in \R}$. 
Then since $(u_t)_{t\in\R}$ is a generalized cocycle for $\sigma^\varphi$ with support projection $f$, by Theorem \ref{connes cocycle existence}, there is a unique faithful normal semifinite weight $\mu_B$ on $fBf$ such that $[D\mu_B,D\varphi_B]_t = u_t$ for all $t\in \R$. Put $\mu:=\mu_B\circ E_B$ on $fMf$ and observe $[D\mu,D\varphi]_t = u_t$ for all $t\in \R$. 
For any $t\in \R$ and $a\in eAe$, using the equation $v u_t = \omega_t \sigma_t^\varphi(v)$ where $\omega_t=[D\psi:D\varphi]_t$, it is easy to compute that
\begin{align*}
	 \sigma_t^\psi(vv^*) =vv^*, \quad 
	 \sigma_t^\mu(v^*v) = v^*v, \quad \text{and}\quad
	 \sigma_t^\mu(\theta(a))  = \theta(\sigma_t^\psi(a)).
\end{align*}
We get that $vv^* \in eM_\psi e$ and $v^*v\in (fMf)_\mu$. 
We extend $\mu$ by $f\mu f + (1-f)\varphi(1-f)$ and still denote by $\mu$. It satisfies that $\mu = \mu \circ E_B$ on $1_BM1_B$ and $1_B,f \in M_\mu$. 
We put $e_0:=vv^* \in eM_\psi e$ and $f_0:=v^*v \in fM_\mu f$. For any $t\in \R$, using Lemma \ref{connes cocycle lemma}, we have
\begin{eqnarray*}
	[D(v\mu v^*):D\varphi]_t
	&=& [D(v\mu v^*):D\mu]_t[D\mu:D\varphi]_t\\
	&=& v \sigma_t^\mu(v^*) [D\mu:D\varphi]_t\\
	&=& v[D\mu:D\varphi]_t \sigma_t^\varphi(v^*) \\
	&=& vu_t \sigma_t^\varphi(v^*) \\
	&=& \omega_t \sigma_t^\varphi(vv^*) \\
	&=& \sigma_t^\psi(v v^*) \omega_t \\
	&=& v v^* \omega_t \\
	&=&  [D(e_0\psi e_0): D\varphi]_t.
\end{eqnarray*}
We get $e_0\psi e_0= v\mu v^*$. Hence $(e,f,\theta,v)$ and $\psi,\mu$ witness $(A,E_A) \preceq_M(B,E_B)$, but $\mu$ is not necessarily bounded. So we have to exchange $\mu$ by a bounded one.

	Since $e_0\psi e_0= v\mu v^*$, it holds that $\mu_B( E_B (f_0)) = \mu(v^*v) =  \psi(e_0) <\infty$. Since $\sigma_t^{\mu_B}(E_B(f_0)) = E_B(\sigma_t^\mu(f_0))=E_B(f_0)$ for all $t\in \R$, and since $f_0=v^*v\in \theta(eAe)'$, $E_B(f_0)$ is contained in $(fBf)_{\mu_B} \cap \theta(eAe)'$. 
Combined with the fact that $v^*vE_B(f_0)  \neq 0$ (because $E_B(v^*vE_B(f_0)) = E_B(f_0)^2 \neq 0$), 
there is a nonzero spectral projection $f'\in (fBf)_{\mu_B}\cap \theta(eAe)'$ of $E_B(f_0)$ such that $vf'\neq 0$ and $\mu_B(f') <\infty$. 
Put $v':=vf'$, $\theta'(a):=\theta(a)f'$ for $a\in eAe$ and $u_t':= f'u_t$ for $t\in \R$. We claim that, up to the polar decomposition of $v'$, $(e,f',\theta',v')$ and $(u_t')_{t\in \R}$ witness $(A,\sigma^{\psi})\preceq _M(B,\sigma^{\varphi})$. 

	It is easy to see that $v'\theta'(a) = a v'$ for all $a\in eAe$, hence $(e,f',\theta',v')$ witnesses $A\preceq_MB$. 
For any $t\in \R$, since $f'=\sigma_t^{\mu}(f')$, one has
	$$(u_t')^*u'_t= u_t^* f'u_t=u_t^* \sigma_t^{\mu}(f')u_t= \sigma_t^\varphi(f').$$
This means $u_t'=f'u_t =u_t \sigma_t^\varphi(f') $ for all $t\in \R$. Using this, for any $a\in eAe$ and $t,s\in \R$, it is easy to compute that 
\begin{align*}
	u'_{t+s} = u_t' \sigma_t^\varphi(u_s'), \quad 
	v' u'_t = \omega_t \sigma_t^\varphi(v') \quad \text{and} \quad
	 u_t'\sigma^\varphi_t(\theta'(a))(u_t')^* =\theta'(\sigma^\psi_t(a)).
\end{align*}
Thus $(e,f',\theta',v')$ and $(u_t')_{t\in \R}$ witness $(A,\sigma^{\psi})\preceq _M(B,\sigma^{\varphi})$.

	We exchange $v'$ with its polar part. Then by using $(e,f',\theta',v')$ and $(u_t')_{t\in \R}$, and by following the same construction as we did for $\mu$, we again construct a faithful normal semifinite weight $\mu'$ on $M$ such that $u_t' = [Df'\mu' f': D\varphi]_t$ for all $t\in \R$, and $e_0'\psi e_0'= v'\mu' \, v'^*$, where $e_0':=v'v'^*$. 
Since  
	$$[Df'\mu' f': D\varphi]_t = u_t' = f' u_t = f'[Df\mu f: D\varphi]_t=[Df'\mu f': D\varphi]_t$$
for all $t\in \R$, it holds that $f'\mu' f' = f'\mu f'$. In particular, since $\mu(f')<\infty$, $f'\mu' f'$ is bounded. By construction, $\mu'$ is bounded on $M$ and hence $(e,f',\theta',v')$ and $\psi,\mu'$ witness $(A,E_A)\preceq _M(B,E_B)$.
\end{proof}

	We record the following permanence property. 

\begin{Lem}\label{intertwining for unital subalgebra}
	Let $D \subset A$ be a unital von Neumann subalgebra with expectation $E_D$. 
\begin{itemize}
	\item[$\rm (1)$] If $(A,\sigma^\psi) \preceq_M (B,\sigma^\varphi)$, then we have $(D,\sigma^{\psi'}) \preceq_M (B,\sigma^\varphi)$ for any faithful $\psi' \in M_*^+$ which is preserved by $E_D \circ E_A$.
	\item[$\rm (2)$] If $(A,E_A) \preceq_M (B,E_B)$, then we have $(D,E_D\circ E_A) \preceq_M (B,E_B)$.
\end{itemize}
\end{Lem}
\begin{proof}
	They are immediate by Lemma \ref{relation for unital and corner of actions}(1) and Theorem \ref{relation for expectations and actions}.
\end{proof}

\subsection*{Proof of Theorem \ref{thmA}}

	Now we prove Theorem \ref{thmA}. We continue to use $A,B \subset M$ with expectations, and we only fix $E_B$. We also fix a type III$_1$ factor $(N,\omega)$ as in the statement of Theorem \ref{thmA}. 

The next lemma is the key observation to prove Theorem \ref{thmA}.

\begin{Lem}\label{key lemma}
	Let $E_A\colon 1_AM1_A\to A$ be a faithful normal conditional expectation, $\psi,\varphi\in M_*$ be faithful states which are preserved by $E_A,E_B$ respectively. The following conditions are equivalent.
\begin{enumerate}
	\item[$\rm(1)$] We have that $(A, E_A)\preceq_{M}  (B, E_B)$.
	\item[$\rm(2)$] We have that $(A\ovt N, E_A\otimes \id_N)\preceq_{M\ovt N}  (B\ovt N, E_B\otimes \id_N)$.
	\item[$\rm(3)$] We have that  $\Pi_{\varphi\otimes\omega,\psi\otimes \omega}(C_{\psi\otimes \omega}(A\ovt N))\preceq_{C_{\varphi\otimes \omega}(M\ovt N)}C_{\varphi\otimes \omega}(B\ovt N)$.
\end{enumerate}
\end{Lem}
\begin{proof}
	(1)$\Rightarrow$(2) This is trivial (one only needs to take tensor products with $1_N$ or $\id_N$). 

	(2) $\Rightarrow$ (3) By Theorem \ref{relation for expectations and actions} and Lemma \ref{relation for unital and corner of actions}(1), item (2) is equivalent to $(A\ovt N, \sigma^{\psi\otimes \omega})\preceq^{\rm uni}_{M\ovt N}  (B\ovt N, \sigma^{\varphi\otimes \omega})$. 
By Theorem \ref{thm unital intertwining with actions}, we get item (3). 

	(3) $\Rightarrow$ (1) We first recall the following general facts (some of which were mentioned in Section \ref{Preliminaries}). 
Since $\langle C_\varphi(M), C_\varphi(\widetilde{B}) \rangle$ is  generated by $\langle M,\widetilde{B} \rangle$ and $L_{\varphi}\R$, and since $\sigma^{\widehat{\varphi}}_t = \Ad(\Delta_{\varphi}^{it})$, where $\widehat{\varphi}=\varphi\circ \widehat{E}_{\widetilde{B}}$, $\langle C_\varphi(M), C_\varphi(\widetilde{B}) \rangle$ is canonically identified as $C_{\widehat{\varphi}}(\langle M,\widetilde{B}\rangle)$. 
Put $\widehat{\psi}:=\psi\circ \widehat{E}_{\widetilde{B}}$. Since it satisfies $[D\widehat{\psi}:D\widehat{\varphi}]_t = [D\psi:D\varphi]_t$ for all $t\in \R$, the map $\Pi_{\widehat{\varphi},\widehat{\psi}}\colon C_{\widehat{\psi}}(\langle M,\widetilde{B}\rangle) \to C_{\widehat{\varphi}}(\langle M,\widetilde{B}\rangle)$ restricts to $\Pi_{{\varphi},{\psi}}\colon C_{\psi}(M)\to C_\varphi(M)$. 
Since $1_B=\pi_{\sigma^\varphi}(1_B)$ is the unit of $C_\varphi(B)$,  for the modular conjugation $J_{C_\varphi(M)}$ on  $L^2(C_\varphi(M))=L^2(M)\otimes L^2(\R)$ (with respect to the dual weight of $\varphi$), it holds that 
	$$J_{C_\varphi(M)} 1_{C_\varphi(B)}J_{C_\varphi(M)}=J_{C_\varphi(M)} 1_{B}J_{C_\varphi(M)} = J 1_B J \otimes 1_{L^2(\R)}.$$
We note that the unitization of $C_\varphi(B)$ is contained in $C_\varphi(\widetilde{B})$, but they are different in general. 
We will use these observations for $A\ovt N, B\ovt N\subset M\ovt N$.

	Now we start the proof.  We put $\mathcal{B}:=C_{\varphi\otimes \omega}(B\ovt N)$, $\mathcal{B}_1:=C_{\varphi\otimes \omega}(\widetilde{B}\ovt N)$, $\mathcal{M}:=C_{\varphi\otimes \omega}(M\ovt N)$, $\mathcal{A}:=C_{\psi\otimes \omega}(A\ovt N)$, and  $\Pi:=\Pi_{\widehat{\varphi\otimes \omega}, \widehat{\psi\otimes \omega}}$, so that our assumption is written as $\Pi(\mathcal A) \preceq_{\mathcal M} \mathcal B$. 
Note that the unitization of $\mathcal{B}$ is contained in $\mathcal{B}_1$.  
Take $(e,f,\theta,v)$ which witnesses $\Pi(\mathcal A) \preceq_{\mathcal M} \mathcal B$. Let $w_i\in \mathcal A$ be partial isometries such that $w_i^* w_i \leq e$ and $\sum_{i}w_iw_i^* = z_{\mathcal A}(e)$, where $z_{\mathcal A}(e)$ is the central support of $e$ in $\mathcal A$. 
Put $d:=\sum_i \Pi(w_i) v e_{\mathcal{B}_1} v^* \Pi(w_i^*)$ and observe that 
	$$d\in \Pi(\mathcal{A})' \cap 1_{\Pi(\mathcal{A})}\langle \mathcal M, \mathcal{B}_1\rangle 1_{\Pi(\mathcal{A})}, \quad d=d \mathcal J 1_{\mathcal{B}}\mathcal J, \quad \text{and} \quad  \widehat{E}_{\mathcal{B}_1}(d) < \infty,$$
where $\mathcal J$ is the modular conjugation for $L^2(\mathcal {M})$. Note that $\mathcal{J}1_{\mathcal{B}}\mathcal{J} = J1_BJ\otimes 1_N\otimes 1_{L^2(\R)}$ as we have explained. 
\begin{claim}
The element $d$ is contained in 
	$$\left[ A'\cap 1_A\langle M,\widetilde{B} \rangle J1_{B}J1_A\right]_{\widehat{\psi}} \otimes \C1_N \otimes \C 1_{L^2(\R)}.$$
\end{claim}
\begin{proof}
	Observe that 
	$$\Pi^{-1}(d)\in \mathcal{A}'\cap 1_{\mathcal{A}}\Pi^{-1}(\langle \mathcal{M},\mathcal{B}_1\rangle \mathcal{J}1_{\mathcal{B}}\mathcal{J})1_{\mathcal{A}}.$$
Observe $\Pi^{-1}(\langle \mathcal{M},\mathcal{B}_1\rangle ) = C_{\widehat{\psi\otimes \omega}}(\langle M\ovt N, \widetilde{B}\ovt N\rangle) $ and  $\widehat{\psi\otimes \omega}=(\psi\otimes \omega)\circ \widehat{E}_{\widetilde{B}\ovt N} = \widehat{\psi}\otimes \omega$. 
Then using $\widehat{\psi}= \psi \circ E_A\circ \widehat{E}_{\widetilde{B}}$ on $1_A\langle M,\widetilde{B} \rangle 1_A$, we can apply  Lemma \ref{III1 factor tensor lemma} (to the inclusion $A \subset 1_A\langle M,\widetilde{B} \rangle 1_A$ with the operator valued weight $E_A\circ \widehat{E}_{\widetilde{B}}$) and get that 
	$$\mathcal{A}'\cap 1_{\mathcal{A}}\Pi^{-1}(\langle \mathcal{M},\mathcal{B}_1 \rangle )1_{\mathcal{A}} = \left[ A'\cap 1_A\langle M,\widetilde{B} \rangle 1_A\right]_{\widehat{\psi}} \otimes \C1_N \otimes \C 1_{L^2(\R)}.$$
Since $\Pi$ is the identity on $\langle M\ovt N,\widetilde{B}\ovt N\rangle$, $d$ is also contained in this set. Finally by multiplying $\mathcal{J}1_{\mathcal{B}}\mathcal{J} = J1_BJ\otimes 1_N\otimes 1_{L^2(\R)}$, we get the conclusion of the claim.
\end{proof}
	By the claim, we can regard that $d$ is contained in $\left[ A'\cap 1_A\langle M,\widetilde{B} \rangle J1_{B}J1_A\right]_{\widehat{\psi}}$. 
As we mentioned in Section \ref{Preliminaries},  $\widehat{E}_{\mathcal{B}_1}$ coincides with $\widehat{E}_{\widetilde{B}\ovt N}\rtimes \R$ (the natural crossed product extension of $\widehat{E}_{\widetilde{B}\ovt N}$), hence the restriction of $\widehat{E}_{\mathcal{B}_1}$ on $\langle M\ovt N, \widetilde{B}\ovt N \rangle$ coincides with $\widehat{E}_{\widetilde{B}\ovt N}$. It then holds that  
	$$\infty > \widehat{E}_{\mathcal{B}_1}(d) =\widehat{E}_{\widetilde{B}\ovt N}(d) = (\widehat{E}_{\widetilde{B}}\otimes \id_N)(d)=\widehat{E}_{\widetilde{B}}(d).$$
Thus $d$ satisfies the condition in Theorem \ref{thm unital intertwining with actions}(4) and we get $(A,\sigma^{\psi})\preceq^{\rm uni}_{M} (B,\sigma^{\varphi})$. 
By Lemma \ref{relation for unital and corner of actions}(1) and Theorem \ref{relation for expectations and actions}, this is equivalent to item (1).
\end{proof}

\begin{proof}[Proof of Theorem \ref{thmA}]
	We first prove the equivalence of the first two conditions. Assume that $A\preceq_MB$. By Lemma \ref{intertwining implies expectation}, there is a projection $q\in A' \cap 1_AM1_A$ and a faithful normal conditional expectation $E_{Aq}\colon qMq\to Aq$ such that $(Aq,E_{Aq})\preceq_M(B,E_B)$. 
Put $A^q := W^*\{A,q\} = Aq \oplus Aq^{\perp}$, where $q^\perp := 1_A -q$. Observe that $Aq^\perp \subset q^\perp Mq^\perp$ is with expectation, say $E_{Aq^\perp}$. 
Then by definition, the condition $(Aq,E_{Aq})\preceq_M(B,E_B)$ implies $(A^{q},E_{Aq}\oplus E_{Aq^{\perp}})\preceq_M(B,E_B)$. 
Since $A \subset 1_AM1_A$ is with expectation,  $A \subset A^q$ is also with expectation. By Lemma \ref{intertwining for unital subalgebra}, it holds that $(A,E_{A})\preceq_M(B,E_B)$ for some faithful normal conditional expectation $E_A \colon 1_AM1_A\to A$. 
By Theorem \ref{relation for expectations and actions}, we get that $(A,\sigma^\psi)\preceq_M(B,\sigma^\varphi)$ for any faithful $\psi\in M_*^+$ which is preserved by $E_A$. 
This finishes the proof of the first part of the theorem.

	We next prove the equivalence of items (1), (2), and (3). The equivalence of items (1) and (2) is proved in Theorem \ref{relation for expectations and actions}. Using Lemma \ref{key lemma}, item (3) is also equivalent.
\end{proof}

\section{Crossed products with groups in the class $\mathcal{C}$}\label{Crossed products on type III factors}

In this section we prove Theorem \ref{thmC}. 
Throughout this section, we will fix an outer action $\Gamma \curvearrowright^\alpha B$ of a discrete group $\Gamma$ on a $\sigma$-finite diffuse factor $B$. We put $M:=B\rtimes_\alpha \Gamma$.

\subsection*{General facts on outer actions}

We first recall several well known facts on outer actions and associated crossed products. 

\begin{Lem}\label{lemma for core of crossed product}
	Let $\varphi$ be a faithful normal state on $M$ which is preserved by $E_B$. 
Then one can define a $\Gamma$-action $\widetilde{\alpha}$ on $C_{\varphi}(B)$ by, for all  $g\in \Gamma$, $b\in B$,  $t\in \R$,
\begin{align*}
	\widetilde{\alpha}_g(b)=\alpha_g(b) \quad \text{and} \quad  \widetilde{\alpha}_g(\lambda_t^{\varphi})=  [D (\varphi\circ \alpha_{g^{-1}}) : D \varphi]_t \lambda_t^{\varphi}.
\end{align*}
We have a canonical identification 
	$$(B\rtimes_{\alpha} \Gamma)\rtimes_{\sigma^\varphi}\R \simeq (B\rtimes_{\sigma^{\varphi}}\R)\rtimes_{\widetilde{\alpha}} \Gamma $$
which is the identity on $B$, $L\Gamma$, and $L_\varphi \R$.
\end{Lem}
\begin{proof}
	This follows by direct computations by using $\Ad(\Sigma)$, where $\Sigma$ is the flip map on $L^2(B)\otimes \ell^2(\Gamma) \otimes L^2(\R)$ for the second and the third components.
\end{proof}

\begin{Lem}\label{lemma for Fourier decomposition}
	Let $p\in B$ be a projection, $B_0\subset pBp$ an irreducible subfactor, and $\beta\colon B_0 \to B_0$ a $\ast$-homomorphism such that $\beta(B_0)' \cap pBp = \C p$. 
Let $x\in pMp$ be any element with the Fourier decomposition $x= \sum_{g\in \Gamma}x_g \lambda_g$. If $x y = \beta(y)x$ for all $y\in B_0$, then we have that 
\begin{itemize}
	\item $x_g \lambda_g y = \beta(y)x_g \lambda_g$ and $x_g \alpha_g(y) = \beta(y)x_g$ for all $y\in B_0$ and $g\in \Gamma$;
	\item $x_g x_g^* \in \C p$ and $x_g^* x_g \in \C \alpha_{g}(p)$;
	\item if $x\in \mathcal{U}(pMp)$ and $B_0' \cap pMp=\C p$, there is a unique $g\in \Gamma$ such that $x= x_g \lambda_g$.
\end{itemize}
\end{Lem}
\begin{proof}
For all $y \in B_0$, we have 
	$$\sum_{g\in \Gamma} x_g \lambda_gy= x y=\beta(y)x = \sum_{g\in \Gamma} \beta(y)x_g \lambda_g.$$
By comparing coeffients, one has $x_g \lambda_g y = \beta(y)x_g \lambda_g$ and $x_g \alpha_g(y) = \beta(y)x_g$ for all $y\in B_0$ and $g\in \Gamma$. 
It holds that $x_g x_g^* = x_g \lambda_g (  x_g \lambda_g)^* \in \beta(B_0)' \cap pBp = B_0' \cap pBp=\C p$, and $\alpha_{g^{-1}}(x_g^* x_g)  =  ( x_g \lambda_g)^*x_g \lambda_g   \in B_0' \cap pBp=\C p$ for all $g\in \Gamma$. 
	Assume further that $x$ is a unitary in $pMp$ and $B_0' \cap pMp=\C p$. Fix $g\in \Gamma$ such that $x_g\neq 0$. Then it holds that
	$$x_g \lambda_g y = \beta(y)x_g \lambda_g = xyx^* x_g \lambda_g,$$
hence $x^*x_g \lambda_g \in B_0' \cap pMp = \C p$. We conclude that $x=x_g\lambda_g$. 
\end{proof}

\begin{Lem}\label{lemma for normal subgroup}
	Let $\Lambda\curvearrowright^\beta A$ be any outer action of a discrete group on a factor. Assume that $M = A \rtimes_\beta\Lambda$ such that $A\subset B$. Then there is a surjective homomorphism $\pi \colon \Lambda \to \Gamma$ such that
\begin{itemize}
	\item  for any $h\in \Lambda$ there is a unique $u_h\in \mathcal{U}(B)$ such that $\lambda^\Lambda_h=u_h \lambda_{\pi(h)}^\Gamma$;
	\item $B=A \rtimes_\beta \ker(\pi)$.
\end{itemize}
In particular, $\beta$ induces a cocycle action $\Lambda/\ker(\pi) \curvearrowright A\rtimes_\beta \ker(\pi)$, and it is cocycle conjugate to $\alpha$ via $A\rtimes_\beta\ker(\pi) = B$ and $\pi\colon \Lambda/\ker(\pi) \simeq \Gamma$.
\end{Lem}
\begin{proof}
	Since $A' \cap M =\C$,  by Lemma \ref{lemma for Fourier decomposition}, any $\lambda_h^\Lambda$ for $h\in \Lambda$ can be uniquely written as $\lambda^\Lambda_h=u_h \lambda_{g}^\Gamma$ for some $g\in \Gamma $ and some $u_h\in \mathcal{U}(B)$. By the uniqueness, if we put $g=\pi(h)$, then $\pi \colon \Lambda \to \Gamma$ define a homomorphism.
Since $A$ and $\lambda^\Lambda_h$ $(h\in \Lambda)$ generate $M$, $B$ and $\pi(\Gamma)$ generate $M$ as well. This implies that $\pi(\Lambda)=\Gamma$ and $\pi$ is surjective. 

	Put $\Lambda_0:=\ker(\pi)$. By construction, $\lambda_h = u_h$ for all $h\in \Lambda_0$ and hence $B_0:=A\rtimes_\beta \Lambda_0 \subset B$. We have to show the opposite inclusion. Let $E_B\colon M \to B$ and $E_{B_0}\colon M \to B_0$ be canonical conditional expectations. Observe that $E_{B_0}\circ E_B = E_{B_0}$. Fix any faithful normal state $\varphi$ on $B_0$ and extend it by $\varphi\circ E_{B_0}$. Then $E_B$ and $E_{B_0}$ extend to Jones projections $e_B$ and $e_{B_0}$ on $L^2(M,\varphi)$. 
Let $x=\sum_{h\in \Lambda} x_h \lambda_h^\Lambda \in A\rtimes_\beta \Lambda$ be any element with the Fourier decomposition. Then we have that 
	$$e_{B}\Lambda_\varphi(x) = \sum_{h\in \Lambda} e_B\Lambda_\varphi(x_h \lambda_h^\Lambda) = \sum_{h\in \Lambda} e_B\Lambda_\varphi(x_h u_h\lambda_{\pi(h)}^\Gamma)= \sum_{h\in \Lambda_0} \Lambda_\varphi(x_h u_h) = \sum_{h\in \Lambda_0} \Lambda_\varphi(x_h \lambda_h^\Lambda)  .$$
Since the last element is contained in $A\rtimes_\beta \Lambda_0$,  we get that $B\subset A\rtimes_\beta\Lambda_0$.

	Put $\widetilde{\Lambda}:=\Lambda/\Lambda_0$ and $\widetilde{A}:=A\rtimes_\beta\Lambda_0$, and fix any section $s\colon \widetilde{\Lambda} \to \Lambda$ such that $s(\Lambda)=e$. 
For any $g,h \in \widetilde{\Lambda}$, we define  $\lambda^{\widetilde{\Lambda}}_{g}:=\lambda^{\Lambda}_{s(g)}$,  $\widetilde{\beta}_{g} := \Ad(\lambda^\Lambda_{s(g)}) \in \mathrm{Aut}(\widetilde{A})$, $\widetilde{u}_g:=u_{s(g)}$,  and $c(g,h):=\lambda^{\Lambda}_{ s(g)s(h)s(gh)^{-1} }\in L\Lambda_0$. 
Then it is easy to check that $(\widetilde{\beta},c)$ defines a cocycle action of $\widetilde{\Lambda}$ on $\widetilde{A}$, and that  $\widetilde{\beta}_{g} =  \Ad(\widetilde{u}_{s(g)})\circ \alpha_{\pi(g)}$ and $ 1 = \widetilde{u}_{g}^* \widetilde{\beta}_{g}(\widetilde{u}_{h}^*) c(g,h) \widetilde{u}_{gh}$ for all $g,h\in \widetilde{\Lambda}$. 
Thus using $\widetilde{A}=B$ and  $\pi \colon \widetilde{\Lambda} \simeq \Gamma$, $(\widetilde{u}_{g})_{g\in \widetilde{\Lambda}}$ gives a cocycle conjugacy between $\widetilde{\Lambda}\curvearrowright^{(\widetilde{\beta},c)} \widetilde{A}$ and $\Gamma \curvearrowright^\alpha B$. 
\end{proof}

\subsection*{Actions of groups in the class $\mathcal C$}

We continue to use the outer action $\Gamma\curvearrowright^\alpha B$ on a $\sigma$-finite diffuse factor and $M=B\rtimes\Gamma$. 
The next proposition is a generalization of \cite[Lemma 8.4]{IPP05}.

\begin{Pro}\label{intertwining for crossed products}
	Let $p\in B$ be a projection and $A\subset pMp$ be a subfactor with expectation such that $A' \cap pMp =\C p$ and $\mathcal{N}_{pMp}(A)''=pMp$. 
\begin{enumerate}
	\item[$\rm (1)$] If $A\preceq_M  B$, then there exist $(e,f,\theta,v)$ witnessing $A\preceq_M  B$ and a finite normal subgroup $K \leq \Gamma$ such that 
	$$\theta(eAe)' \cap fBf = \C f, \quad vv^*=e, \quad v^*v \in \theta(eAe)' \cap f(B\rtimes K)f  .$$
Assume further that $\Gamma$ has no finite normal subgroups, and that either both of $A,B$ are of type II$_1$ or both are properly infinite. Then we can choose $e=f=p$ and $v\in \mathcal{U}(pMp)$.

	\item[$\rm (2)$] Assume that $p=1$ and that $A$ has a decomposition $M=A\rtimes \Lambda$ for some outer action of a discrete group $\Lambda$ on $A$. Assume that $\Gamma$ and $\Lambda$ are ICC. 
If $A\preceq_M  B$ and $B\preceq_M  A$, then $A$ and $B$ are unitarily conjugate in $M$.
\end{enumerate}
\end{Pro}
\begin{proof}
	(1) Since $B$ is a factor, using \cite[Remark 4.5]{HI15}, we may assume that $A\preceq_{M} pBp$. 
We first show that, using the assumption $A' \cap pMp =\C p$, there is $(e,f,\theta,v)$ which witnesses $A\preceq_M pBp$ such that $\theta(eAe) \subset fBf$ is irreducible. 

	Since $vv^* \in (eAe)' \cap eMe=\C e$, one has $vv^*=e$ and moreover $v^*v$  is a minimal projection in $\theta(eAe)' \cap fMf$. Indeed, for any projection $r \leq v^*v$ in $\theta(eAe)' \cap fMf$, $vrv^*\in (eAe)' \cap eMe=\C e$ is again $e$, hence $r=vv^*$. 
We may assume that the support projection of $E_B(v^*v)$, which is contained in $\theta(eAe)' \cap fBf$, coincides with $f$. 
Let $z$ be the central support projection of $v^*v$ in $\theta(eAe)' \cap fMf$. Then since $v^*v$ is minimal, $(\theta(eAe)' \cap fMf)z$ is a type I factor. 
Since $\theta(eAe) \subset fBf$ is with expectation, so is  the inclusion $\theta(eAe)' \cap fBf \subset \theta(eAe) ' \cap fMf$. In particular, $(\theta(eAe)' \cap fBf)z$ is an atomic von Neumann algebra. 
Since $z$ commutes with $\theta(eAe)' \cap fBf$, there is a unique projection $w \in \mathcal{Z}(\theta(eAe)' \cap fBf)$ such that $(\theta(eAe)' \cap fBf)w \ni aw  \mapsto az \in (\theta(eAe)' \cap fBf)z$ is isomorphic. 
Thus there is a minimal projection $q$ in $\theta(eAe)' \cap fBf$. Since $q \leq f$, $q$ is smaller than the support of $E_B(v^*v)$, hence $vq \neq 0$. 
Now  $(e,q,\theta(\cdot)q,vq)$ witness $A\preceq_M pBp$ (up to the polar decomposition of $vq$) and satisfies that  $\theta(eAe)q \subset qBq$ is an irreducible inclusion. 

	Thus we can start the proof by assuming $\theta(eAe)' \cap fBf=\C f$. Put $B_0:=\theta(eAe) \subset fBf$ and note that $B_0 ' \cap fBf=\C f$. Consider the Fourier decomposition $q:=v^*v = \sum_{g\in \Gamma} x_g \lambda_g \in B\rtimes \Gamma$. 
Since $q \in B_0' \cap fMf$, by Lemma  \ref{lemma for Fourier decomposition}, it holds that $x_g \lambda_g \in B_0' \cap fMf$, $x_gx_g^* = \C f$, and $x_g^*x_g  \in \C \alpha_{g}(f)$. 
Define subgroups $K, \Gamma_0 \leq \Gamma$ by
\begin{align*}
	K:=& \{ g\in \Gamma \mid \Ad(w_g)\circ \alpha_g|_{B_0} = \id_{B_0} \text{ for some } w_g\in B \text{ s.t.\ } w_gw_g^* =f, \ w_g^*w_g = \alpha_g(f)  \};\\
	\Gamma_0:=& \{ g\in \Gamma \mid \Ad(w_g)\circ \alpha_g(B_0) = B_0 \text{ for some } w_g\in B \text{ s.t.\ } w_gw_g^* =f, \ w_g^*w_g = \alpha_g(f)\}.
\end{align*}
By definition, $q$ is contained in $B\rtimes K$ and $K$ is a normal subgroup of $\Gamma_0$. We will prove that $|K|<\infty$ and $\Gamma_0=\Gamma$.

	We claim that $K$ is a finite group. Fix $(w_g)_{g\in K}$ which appeared in the definition of $K$ such that $w_e=1$. For all $g,h\in K$, define 
	$$\alpha^w_g :=\Ad(w_g)\circ \alpha_g \quad \text{and} \quad \mu_{g,h}:=  w_{g}\alpha_g(w_h)w_{gh}^* \in \mathcal{U}(fBf) $$
and observe that $(\alpha^w,\mu)$ gives a cocycle action of $K$ on $fBf$, so that $f(B\rtimes_\alpha K)f = fBf \rtimes_{(\alpha^w,\mu)}K$. The condition $\alpha^w|_{B_0} = \id_{B_0}$ implies that $\mu_{g,h} \in \C f$ for all $g,h\in K$, hence we can regard $\mu$ as a scalar 2-cocycle. In particular $fBf \rtimes_{(\alpha^w,\mu)}K$ contains a finite von Neumann algebra $(\C f)\rtimes_{(\alpha^w,\mu)}K$. 
Since $B_0 ' \cap fBf = \C f$ and $\alpha^w|_{B_0} = \id_{B_0}$, using Fourier decompositions, it is easy to see that 
	$$ B_0' \cap [fBf \rtimes_{(\alpha^w,\mu)}K] = (\C f) \rtimes_{(\alpha^w,\mu)}K.$$
The left hand side contains the minimal projection $q$, and hence so does the right hand side. This implies that $K$ is a finite group. (Indeed if infinite, one has a sequence of unitaries which converges weakly to 0, but it is impossible in a finite von Neumann algebra with a minimal projection.)

	We next claim that $\Gamma=\Gamma_0$. Observe that  $eAe\subset e(B\rtimes \Gamma)e$ is regular and $eAe$ is a diffuse factor. 
Since $\Ad(v^*)$ is an isomorphism between $eAe\subset e(B\rtimes \Gamma)e$ and  $B_0q \subset q(B\rtimes \Gamma)q$, it holds that $B_0q \subset q(B\rtimes \Gamma)q$ is regular. 
	Fix $u\in \mathcal{N}_{q (B\rtimes \Gamma)q}(B_0q)$ and consider the Fourier decomposition $u = \sum_{g\in \Gamma} x_g \lambda_g \in B \rtimes \Gamma$. 
Since $\Ad(u)$ is an isomorphism on $B_0q$, using $B_0q \simeq B_0$, we can define $\beta^u \in \mathrm{Aut}(B_0)$ by $\beta^u(y)q=uyu^*$ for all $y\in B_0$. 
By Lemma  \ref{lemma for Fourier decomposition}, we get that for all $y\in B_0$ and $g\in \Gamma$, 
	$$x_g \lambda_g y = \beta^u(y)x_g \lambda_g, \quad x_g x_g^* \in \C f, \quad \text{and} \quad x_g^* x_g \in \C \alpha_g(f).$$
So each $x_g \in fB\alpha_g(f)$ is a scalar multiple of a partial isomrtry. 
Observe that $\Ad(x_g\lambda_g)(y) = \beta^u(y)x_g x_g^* \in \beta^u(B_0)=B_0$ for all $y\in B_0$, so $\Ad(x_g\lambda_g)$ preserves $B_0$. By definition, this means that if $x_g\neq 0$, then $g \in \Gamma_0$. Hence it holds that $u\in q(B\rtimes \Gamma_0)q$. Since $Bq \subset q(B\rtimes \Gamma)q$ is regular, we conclude that $q(B\rtimes \Gamma)q=q(B\rtimes \Gamma_0)q$. 
Since $q \in B\rtimes \Gamma_0$ and since $B\rtimes \Gamma_0$ is a diffuse factor, we indeed have that $B\rtimes \Gamma= B\rtimes \Gamma_0$. This means that $\Gamma= \Gamma_0$. 

	Finally assume that $\Gamma$ has no finite normal subgroups. Then $K$ must be trivial, so $v^*v \in B$ and we may assume $f=v^*v$. We have that there is a partial isometry $v\in pMp$ such that $vv^*=e\in A$, $v^*v =f\in pBp$, and $v^*Av \subset fBf$. 
If both of $A,B$ are II$_1$ factors or if both of $A,B$ are properly infinite, then (up to exchanging $e,f$ by smaller ones if necessarily,) we can apply a usual patching method, and obtain that $e=f=p$ and $v\in \mathcal{U}(pMp)$. This is the conclusion.

	(2) Observe that, since $A\rtimes \Lambda=M = B\rtimes \Gamma$, $A$ is a $\rm II_1$ factor if and only if so is $B$. Hence using item (1) of this proposition, we can find $v,w\in \mathcal{U}(M)$ such that $vAv^* \subset B$ and $wBw^* \subset A$. 
Put $u:=vw$ and observe that $uBu^* \subset B$ and $(uBu^*)' \cap B \subset (uBu^*)' \cap M = u(B' \cap M)u^* =\C$. 
By Lemma \ref{lemma for Fourier decomposition}, we can write $u=x_g \lambda_g$ for some $g\in \Gamma $ and $x_g \in \mathcal{U}(B)$. In particular we have $B = uBu^* = vw B w^* v^* \subset v A v^* \subset B$. We conclude that $vAv^*=B$.
\end{proof}

The next lemma explains how we use the property of the class $\mathcal C$ for actions on type III factors. This uses our Theorem \ref{thmA}.

\begin{Lem}\label{lemma for class C}
	Let $p\in M$ be a projection, and $A \subset pMp$ be a subfactor with expectation $E_A$. Assume that $\Gamma$ is in the class $\mathcal{C}$, $A' \cap pMp = \C$, $A$ is amenable, and $\mathcal{N}_{pMp}(A)'' \subset pMp$ has finite index. Then we have $A \preceq_M  B$.
\end{Lem}
\begin{proof}
	Put $P:=\mathcal{N}_{pMp}(A)''$ and let $N$ be the hyperfinite type $\rm III_1$ factor and $\omega$ a faithful normal state such that $N_\omega' \cap N =\C$. 
Let $E_A,E_P$ be any faithful normal conditional expectations for $A,P$ respectively. Observe that the condition $A' \cap pMp\subset A$ implies that normal expectations onto $A$ and $P$ are unique, hence $E_A\circ E_P=E_A$. 
Using this uniqueness and using Theorem \ref{thmA}, there exist $\psi,\varphi$, which are preserved by $E_A,E_B$ respectively such that
	$$\Pi_{\varphi\otimes \omega,\psi\otimes \omega}(C_{\psi\otimes \omega}(A\ovt N)) \not \preceq_{C_{\varphi\otimes \omega}(M\ovt N)}C_{\varphi\otimes \omega}(B\ovt N).$$ 
There is a canonical inclusion $C_{\psi\otimes \omega}(A\ovt N) \subset C_{\psi\otimes \omega}(P\ovt N)$, which is regular by \cite[Lemma 4.1]{BHV15}. 
For notation simplicity, we omit $\Pi_{\varphi\otimes \omega,\psi\otimes \omega}$ and write as $\mathcal{M}:=C_{\varphi\otimes \omega}(M\ovt N)$,  $\mathcal{B}:=C_{\varphi\otimes \omega}(B\ovt N)$,  $\mathcal{A}:=C_{\psi\otimes \omega}(A\ovt N)$, and $\mathcal{P}:=C_{\psi\otimes \omega}(P\ovt N)$. 
Observe that $\mathcal{A}$ is amenable and $\mathcal{P} \subset \mathcal{M}$ has finite index. 

	By Lemma \ref{lemma for core of crossed product}, there is an identification $\mathcal{M} = \mathcal B \rtimes_{\widetilde{\alpha}}\Gamma$. 
Let $r\in L_{\varphi\otimes \omega}\R$ be any projection such that $\Tr_{\varphi\otimes \omega}(r)<\infty$. Then since $\mathcal{B}$ is a type II$_\infty $ factor and since $\widetilde{\alpha}$ preserves the canonical trace on $\mathcal B$, $r\mathcal {M}r$ is realized as a cocycle crossed product $r\mathcal{B}r \rtimes_{(\widetilde{\alpha}^r,u)} \Gamma$ for some 2-cocycle $u\colon \Gamma \times \Gamma \to r\mathcal{B}r$. 
Since $\mathcal{M}$ is a II$_\infty$ factor, $p$ is infinite, and $r$ is finite, there is $v\in \mathcal{M}$ such that $vv^* = r$ and $p_0:=v^*v \in p\mathcal{A}p$. Put $\mathcal A^v:= v\mathcal A v^*$. 
Observe that $\mathcal{A}^v$ is amenable and that $(\mathcal{A}^v)' \cap r\mathcal{M}r=\C r$ (use Lemma \ref{III1 factor tensor lemma}). 
Since $\mathcal{A}$ is a II$_\infty$ factor, it holds that $p_0\mathcal{N}_{p\mathcal Mp}(\mathcal A)''p_0 = \mathcal{N}_{p_0\mathcal Mp_0}(p_0\mathcal Ap_0)''$. In particular $\mathcal{N}_{r\mathcal{M}r}(\mathcal{A}^v)'' \subset r\mathcal{M}r$ has finite index. 
Hence by the definition of the class $\mathcal C$, we have  $\mathcal{A}^v\preceq_{r\mathcal{M}r}r\mathcal{B}r$. This implies $\mathcal{A}\preceq_{\mathcal{M}}\mathcal{B}$ and hence by Theorem \ref{thmA}, we obtain $A \preceq_M  B$.
\end{proof}

\begin{proof}[Proof of Theorem \ref{thmC}]
	By Lemma \ref{lemma for class C}, we have $A\preceq_M B$. Observe that, $A$ is a type $\rm II_1$ factor if and only if so is $B$. 
Hence we can apply Proposition \ref{intertwining for crossed products}, and find a unitary $u\in \mathcal{U}(M)$ such that $uAu^* \subset B$.  Thus we may assume that $A \subset B$. 
We then apply Lemma \ref{lemma for normal subgroup} and get the conclusion. Note that $\ker(\pi)$ is amenable since $A\rtimes \ker(\pi)$ is amenable and $A$ is a factor.
\end{proof}

\section{Rigidity of Bernoulli shift actions}

In this section, we will study Bernoulli shift actions with type III base algebras. We particularly prove Theorem \ref{thmB} and Proposition \ref{thmE}.

\subsection*{Popa's criterion for cocycle superrigidity}

	The next proposition is a variant of Popa's theorem which was used to prove cocycle superrigidity \cite{Po04,Po05a,Po05b}. See also \cite[Theorem 7.1]{VV14}. 

\begin{Pro}\label{prop for cocycle superrigidity}
	Let $G$ be a locally compact second countable group, $G_1\leq G$ a closed normal subgroup, $(P,\varphi)$ a von Neumann algebra with a faithful normal state. 
Let $G \curvearrowright^\alpha (P,\varphi)$ be a state preserving continuous action. Let $\omega\colon G \to \mathcal{U}(P)$ be a $\sigma$-strongly continuous map such that $\beta_g:=\Ad(\omega_g)\circ \alpha_g$ and $v(g,h):=\omega_g \alpha_g(\omega_h) \omega_{gh}^*$ for $g,h\in G$ define a cocycle action of $G$. 
Assume that 
\begin{itemize}
	\item $v(g,h)=1=v(h,g)$ for all $g\in G_1$ and $h\in G$ (hence $\beta|_{G_1}$ is a genuine action);
	\item there is a faithful state $\psi\in P_*$ which is preserved by $\beta|_{G_1}$;
	\item $(\C p , \beta|_{G_1}) \preceq^{\rm uni}_{P} (\C1_P ,\alpha|_{G_1})$ for all projections $p\in P^\beta$; 
	\item $\alpha|_{G_1}$ is weakly mixing.
\end{itemize}
Then there exist a separable Hilbert space $H$, a projection $f\in \B(H)$, a $\sigma$-strongly continuous map $u\colon G \to \mathcal{U}(f\B(H)f)$, a partial isometry $w\in P\ovt \B(H)$ such that 
	$$w^*w = f, \quad ww^* = 1\otimes e_{1,1}, \quad \text{and} \quad wu_g =  (w_g\otimes 1_H )(\alpha_g\otimes \id_H) (w) \quad \text{for all }g\in G,$$
where $e_{1,1} $ is a minimal projection in $\B(H)$. 
In particular, $(\Ad(u_g))_{g\in G}$ and $(u_gu_h u_{gh}^*)_{g,h\in G}$ define a cocycle action on $f\B(H)f$, and $\beta$ is conjugate to the cocycle action $(\alpha_g\otimes \Ad(u_g))_{g\in G}$ by $w$:
	$$\beta_g(wxw^*)=\alpha_g^\omega( wxw^*) = w   (\alpha_g\otimes \Ad(u_g))(x) w^*, \quad \text{for all }x\in P\ovt f\B(H)f.$$
\end{Pro}
\begin{proof}
	Since most of proofs are straightforward adaptations of \cite[Theorem 7.1]{VV14}, we give only a sketch of the proof.
Take $(H,f,\pi,w)$ and $(u_g)_{g\in G_1}$ which witness $(\C p , \beta|_{G_1})\preceq_{P}(\C 1_P, \alpha|_{G_1})$ (and $H$ can be finite dimensional). 
Observe that $w^*w \in (P\ovt \B(H))^{\alpha\otimes \Ad(u)|_{G_1}} = \C 1_P \ovt \B(H) $ (because $\alpha|_{G_1}$ is weakly mixing), hence  up to exchanging $f$ by $w^*w$, we may assume that $w^*w = f$.

	Thus the condition $(\C p , \beta|_{G_1})\preceq_{P}(\C 1_P, \alpha|_{G_1})$ means that there exist $(n,f,w,u)$: a projection $f\in \M_n$, a continuous homomorphism $u \colon G_1 \to \mathcal{U}(f\M_nf)$, and a partial isometry $w\in (p\otimes e_{1,1})(P\otimes \M_n)f$ such that  $wu_g =  (\omega_g\otimes 1_n )(\alpha_g\otimes \id_n) (w)$ for all $g\in G_1$. 
\begin{claim}
	There exist a separable Hilbert space $H$, a projection $f\in \B(H)$, a partial isometry $w\in P\ovt \B(H)$, and a continuous homomorphism $u\colon G_1 \to \mathcal{U}(f\B(H)f)$ such that 
	\begin{itemize}
		\item $wu_g =  (\omega_g\otimes 1_H )(\alpha_g\otimes \id_H) (w)$ for all $g\in G_1$;
		\item $w^*w = f$ and $ww^* \in pP^\beta p\ovt \C e_{1,1}$, where $e_{1,1}$ is a fixed minimal projection;
		\item there exist finite rank projections $(P_k)_{k\in \N}$ in $\B(H)$ such that $P_k \to 1_H$ as $k\to \infty$ and that each $P_k$ commutes with $u_g$ for all $g\in G_1$.
	\end{itemize}
\end{claim}
\begin{proof}
	Let $\mathcal E$ denote the set of all nonzero projections $e\in P (=P\otimes \C e_{1,1})$ such that there exist $(n,f,w,u)$ which witnesses $(\C p , \beta|_{G_1})\preceq_{P}(\C 1_P, \alpha|_{G_1})$ with $e = ww^*$. 
Then it is straightforward to check that $\mathcal E$ is closed under the following operations:  $\alpha_h(e)\in\mathcal E$ for all $h\in G$ and for all $e\in \mathcal{E}$; $e\vee f\in\mathcal E$ for all $e,f \in \mathcal{E}$; and $e_0 \in \mathcal E$ for all projections $e_0 \in eP^{\beta|_{G_1}}e$ and $e\in \mathcal{E}$. 

	Fix any countable dense subset $X \subset G$. Observe that $\sup_{h\in X} \alpha_h(e) \in pP^{\beta}p$ is realized as a (countably) infinite direct sum of projections in $\mathcal E$, that is, there is a family $(n_i,f_i,w_i,u^i)_{i\in I}$ such that $\sum_{i\in I}w_iw_i^* = \sup_{h\in X} \alpha_h(e)$, where $I$ is a countable set. 
By defining $H:= \bigoplus_{i\in I} \C^{n_i}$, $f:=\bigoplus_{i\in I} f_i$, $w=[w_i]_{i\in I} \in (p\otimes e_{1,1})(B\ovt \B(H))f$, and $u:=\bigoplus_{i\in I}u^i$, we get the conclusion. 
\end{proof}

	Now we define $\mathcal F$ as the set of all nonzero projections $e\in P^\beta (= P^\beta \otimes \C e_{1,1})$ such that there exists $(H,f,w,u)$ which witnesses the conclusion of the claim above with $e=ww^*$. 
Now using the assumption $(\C p , \beta|_{G_1})\preceq_P (\C 1_P , \alpha|_{G_1})$ for all $p \in P^\beta$ and applying a maximality argument, there is a family $(H_i,f_i,w_i,u^i)_{i\in I}$ such that $\sum_{i\in I}w_iw_i^* = 1_P (=1_P \otimes e_{1,1})$, where $I$ is a countable set.  
Define $(H,f,w,u)$ as a direct sum of all $(H_i,f_i,w_i,u^i)_{i\in I}$ (with $w=[w_i]_{i\in I} \in (1\otimes e_{1,1})(B\ovt \B(H))$), and then it satisfies all the conditions in the claim above with $ww^* =1 \otimes e_{1,1}$. 
Hence $(H,f,w,u)$ satisfies the conclusion of this theorem but only for $G_1$. 

	We have to extend the conditions on $G_1$ to that on $G$, using the weak mixingness of $\alpha|_{G_1}$. 
Put $\omega_g^H:=\omega_g\otimes 1_H$, $\alpha_g^H:=\alpha_g \otimes \id_H$,  $\beta_g^H:=\beta_g \otimes \id_H$, and $v^H(g,h):=v(g,h)\otimes 1_H$ for all $g,h \in G$. Extend the map $u$ to the one on $G$ by
	$$ u_g := w^* \omega_g^H\alpha_g^H (w), \quad \text{for all } g\in G.$$ 
It is easy to compute that for any $g,h\in G$,
	$$u_gu_g^* = f = u_g^*u_g \quad \text{and}\quad u_{g}\alpha^H_g(u_h)=w^*   v^H(g,h)w u_{gh}. $$ 
In particular, $u\colon G \to \mathcal{U}(P\ovt f\B(H)f)$ is a cocycle for $\alpha^H$ with a 2-cocycle $w^*v^H(\cdot ,\cdot ) w$. 
To finish the proof, we have only to show that $u$ is a map into $f\B(H)f$, so that $\alpha_g^H(u_h) = u_h$ and $u_gu_h u_{gh}^*=w^*v^H(g,h)w \in f\B(H)f$ for all $g,h \in G$. 

	Fix $g\in G$ and $k\in \N$. Put $H_k:=P_kH$ and $u_h^k:= P_k u_hP_k$ for all $h\in G$, where $(P_n)_{n\in \N}$ is a family of finite rank projections as in the claim (and we regard $P_k = 1_P \otimes P_k$). Then since $P_k$ commutes with $u_h$ for all $h\in G_1$, putting $\alpha_h^u:=\Ad(u_h)\circ \alpha_h$, it holds that 
	$$\alpha_h^u(u_g^k)= P_k\alpha_h^u(u_g)P_k = u_{g}^k u^k_{g^{-1}hg}(u_h^k)^* \in u_g^k \B(H_k) , \quad \text{for all }h\in G_1.$$
Observe that $\alpha_h^u$ is of the form that $\alpha_h\otimes \Ad(u_h)$ for all $h\in G_1$. Then combining the weak mixingness of $\alpha|_{G_1}$ with $(\alpha_h\otimes \Ad(u_h^k))(u_g^k) \in u_g^k \B(H_k)$ for all $h\in G_1$, it holds that $u_g^k \in \B(H_k)$. Since $k$ is arbitrary, we obtain that $u_g \in \B(H)$ as required.
\end{proof}

\subsection*{Rigidity of Bernoulli shifts for cocycle actions}

	Let $\Gamma$ be a countable discrete group, $B_0$ an amenable von Neumann algebra with separable predual, $\varphi_0$ a faithful normal state on $B_0$, and $\Gamma \curvearrowright^\alpha \bigotimes_{\Gamma}(B_0,\varphi_0)=:(B,\varphi)$ the Bernoulli shift action. Put $M:=B\rtimes_\alpha \Gamma$. 
Here we recall the following fact.

\begin{Thm}\label{Bernoulli theorem2}
	Let $p \in M$ be a projection and $A \subset p Mp$ a von Neumann subalgebra with expectation $E_A$. Fix a faithful $\psi \in M_*$ which is preserved by $E_A$, and ${P}:={A}' \cap p {M}_\psi p$. 
If $C_\psi({A}) \not\preceq _{C_\varphi(M)} C_\varphi(L \Gamma)$, then $P$ has an amenable direct summand.
\end{Thm}
\begin{proof}
	This can be proved by applying arguments in \cite[Theorem 4.1]{CPS11}, which is based on the arguments in \cite{Po03,Po04,Po06a} (together with the deformation given in \cite{Io06}). 
Actually one has to modify the spectral gap argument \cite{Po06a} as follows. Put $\widetilde B:= \bigotimes_{\Gamma}(B_0*L\Z,\varphi_0*\tau_{L\Z})$ and extend $\varphi$ and $\alpha$ on $\widetilde{B}$, so that there are canonical inclusions $M \subset \widetilde{B}\rtimes_\alpha \Gamma =:\widetilde{M}$ and $C_\varphi(M) \subset C_\varphi(\widetilde{M})$. Then we can prove the following weak containment: 
	$$ {}_M L^2(C_\varphi(\widetilde{M}))\ominus L^2(C_\varphi(M))_{C_\varphi(M)} \prec {}_M L^2(C_\varphi(M))\otimes L^2(C_\varphi(M))_{C_\varphi(M)} $$
(e.g.\ see the proof of \cite[Theorem 5.2]{Ma16}). Then using the spectral gap argument given in \cite[Lemma 4.1]{Ma16}, we can follow the proof of \cite[Theorem 4.1]{CPS11}.
\end{proof}

\begin{proof}[Proof of Theorem \ref{thmB}]
	Put $M:=B\rtimes_\alpha \Gamma = A\rtimes_\beta \Lambda$. Using Lemma \ref{lemma for class C} and Proposition \ref{intertwining for crossed products}, we may assume $A \subset B$. 
Then by Lemma \ref{lemma for normal subgroup}, there is a surjective homomorphism $\pi \colon \Lambda \to \Gamma$ such that $A\rtimes_\beta \Lambda_0 = B$, where $\Lambda_0:=\ker \pi$, and for any $h\in \Lambda$, there is a unique $u_h \in \mathcal{U}(B)$ such that $\lambda_h^\Lambda = u_h \lambda_{\pi(h)}^\Gamma$. 
Put $\widetilde{A}:=A\rtimes_\beta \Lambda_0$ and  $\widetilde{\Lambda}:=\Lambda/\Lambda_0$. 
Using a fixed section $s\colon \widetilde{\Lambda} \to \Lambda$ such that $s(\Lambda_0)$ is the unit, we will use the following notation: for all $g,h\in \widetilde{\Lambda}$, 
$\widetilde{\beta}_g:=\Ad(\lambda^{\Lambda}_{s(g)}) \in \Aut(\widetilde{A})$,  $c(g,h):=\lambda^{\Lambda}_{ s(g)s(h)s(gh)^{-1} }$, $\lambda_g^{\widetilde{\Lambda}} :=\lambda_{s(g)}^\Lambda$, and  $u_g:=u_{s(g)}$. 
We have a cocycle action $\widetilde{\Lambda}\curvearrowright^{(\widetilde{\beta},c)} \widetilde{A}$ with relations 
	$$\lambda_h^{\widetilde{\Lambda}}= u_g \lambda_{\pi(h)}^\Gamma, \quad \Ad(u_g)\circ \alpha_{\pi(g)}  = \widetilde{\beta}_g, \quad c(g,h) =\widetilde{u}_{g}\alpha_{g}(\widetilde{u}_{h}) \widetilde{u}_{gh}^* \quad \text{for all }g,h\in \widetilde{\Lambda}. $$
	For simplicity we identify $C_\psi(M)=C_\varphi(M)$. Then using Lemma \ref{lemma for core of crossed product}, there is an inclusion
		$$ L_\psi \R \subset C_\psi(\widetilde{A}\rtimes_{\widetilde{\beta}}\widetilde{\Lambda})  = C_\varphi(M) = C_\varphi(B)\rtimes_\alpha \Gamma.$$
Observe that, since $\widetilde{\beta}$ is $\psi$-preserving, $(L_\psi\R)' \cap C_\varphi(M)$ contains a copy of $L\widetilde{\Lambda}$ with expectation, hence $(L_\psi\R)' \cap C_\varphi(M)$ has no amenable direct summand. 
\begin{claim}
We have $(\C p, \sigma^\psi) p \preceq_{B} (\C 1_B , \sigma^\varphi)$ for all projections $p\in B_{\psi}^{\widetilde{\beta}} $.
\end{claim}
\begin{proof}[Proof of Claim]
	Fix any projection $p\in B_{\psi}^{\widetilde{\beta}} $. 
Since $L\widetilde{\Lambda} p$ has no amenable summand, by applying Theorem \ref{Bernoulli theorem2} to $L_\psi\R p$, we obtain that $L_\psi \R p \preceq_{C_\varphi(M)}C_\varphi(L\Gamma)$. 
By Theorem \ref{thm unital intertwining with actions}, to prove this claim, we have only to show that $L_\psi\R p \preceq_{C_\varphi(B)} L_\varphi \R$.

	Suppose by contradiction that $L_\psi\R p\not\preceq_{C_\varphi(B)} L_\varphi \R$. Take a net $(u_i)_i$ in $\mathcal{U}(L_\psi \R)$ such that 
	$$ E_{L_\varphi \R} (b^* u_i p a) \to 0 , \quad \text{for all }a,b \in C_\varphi(B).$$
Observe that for all $h\in \widetilde{\Lambda}$ and $u_i\in L_\psi\R$, since $u_i$ commutes with $\lambda_h^{\widetilde{\Lambda}}$, 
	$$ \lambda^\Gamma_{\pi(h)} u_i p (\lambda^\Gamma_{\pi(h)})^*={u}_{h}^*\lambda_h^{\widetilde{\Lambda}} u_i p(\lambda^{\widetilde{\Lambda}}_{h})^*{u}_{h}={u}_{h}^* u_i p {u}_{h}.$$
It holds that for all $a,b \in C_\varphi(B)$ and $g,h\in \widetilde{\Lambda}$,
\begin{align*}
	E_{C_\varphi(L\Gamma)} ( b \lambda_{\pi(h)}^\Gamma u_ip a\lambda_{\pi(g)}^\Gamma)
	&= 	E_{C_\varphi(L\Gamma)} ( b \left[\lambda_{\pi(h)}^\Gamma u_ip (\lambda_{\pi(h)}^\Gamma)^* \right] \alpha_{\pi(h)}(a)\lambda_{\pi(hg)}^\Gamma)\\
	&= 	E_{C_\varphi(L\Gamma)} ( b \left[{u}_{h}^* u_i p {u}_{h} \right] \alpha_{\pi(h)}(a)\lambda_{\pi(hg)}^\Gamma)\\
	&= 	E_{L_\varphi \R} ( b {u}_{h}^* u_i p {u}_{h} \alpha_{\pi(h)}(a)) \lambda_{\pi(hg)}^\Gamma \to 0.
\end{align*}
By \cite[Theorem 4.3(5)]{HI15}, we get $L_\psi\R p\not\preceq_{C_\varphi(M)} C_\varphi(L\Gamma)$, a contradiction. 
\end{proof}

	Put $G:=\Gamma \times \R$. Since $\alpha$ and $\sigma^\varphi$ commute, we can define a continuous action $G\curvearrowright^{\alpha^\varphi} (B,\varphi)$ by
	$$\alpha^\varphi_{(g,t)}:= \alpha_g\circ\sigma_t^\varphi = \sigma_t^\varphi \circ \alpha_g, \quad \text{for all }(g,t)\in G $$ 
The condition $B_\varphi=\C$ then means that $\alpha^\varphi|_{\R}$ is weakly mixing. 
In the same say, we can define a continuous cocycle action $\widetilde{\Lambda}\times \R \curvearrowright^{\widetilde{\beta}^\psi} (\widetilde{A},\psi)$ with the 2-cocycle $c^\psi((g,t), (h,s)) := c(g,h)$ for all $(g,t), (h,s)\in \widetilde{\Lambda}\times \R$. 

\begin{claim}
	Identify $\widetilde{\Lambda} = \Gamma$ and $\widetilde{A} =B$. Define a $\sigma$-strongly continuous map $\omega\colon G \to \mathcal{U}(B)$ by 
	$$\omega_{(g,t)}:=[D\psi:D\varphi]_t\sigma_t^{\varphi}(u_g)=\sigma_t^{\psi}(u_g) [D\psi:D\varphi]_t , \quad g\in \Gamma, \ t\in \R.$$ 
Then $\omega$ gives a cocycle conjugacy between $\alpha^\varphi$ and $\widetilde{\beta}^\psi$: for all $(g,t),(h,s)\in G$,
	$$ \Ad(\omega_{(g,t)})\circ \alpha_{(g,t)}^\varphi = \widetilde{\beta}_{(g,t)}^\psi \quad \text{and} \quad \omega_{(g,t)} \alpha^\varphi_{(g,t)} (\omega_{(h,s)}) = c^\psi((g,t),(h,s))\omega_{(gh,t+s)}.$$
\end{claim}
\begin{proof}[Proof of Claim]
	Observe that for any $(g,t)\in G$, since $\lambda_t^\varphi$ and $\lambda_g^\alpha$ commute in $C_\varphi(M)$, 
\begin{align*}
	 &\lambda_g^\alpha  \lambda_t^\varphi 
	= u_g^* \lambda_g^{\widetilde{\beta}} [D\varphi: D\psi]_t \lambda_t^\psi 
	= u_g^* \widetilde{\beta}_g([D\varphi: D\psi]_t) \lambda_g^{\widetilde{\beta}}  \lambda_t^\psi   \\
	= \ & \lambda_t^\varphi \lambda_g^\alpha 
	= [D\varphi: D\psi]_t \lambda_t^\psi u_g^* \lambda_g^\beta 
	= [D\varphi: D\psi]_t \sigma_t^\psi(u_g^*)  \lambda_t^\psi \lambda_g^{\widetilde{\beta}} .
\end{align*}
Since $\lambda_t^\psi \lambda_g^{\widetilde{\beta}} = \lambda_g^{\widetilde{\beta}} \lambda_t^\psi $, using $[D\varphi: D\psi]_t^*=[D\psi: D\varphi]_t$, we get that 
	$$\omega_{(g,t)}=\sigma_t^\psi(u_g)[D\psi: D\varphi]_t  = \widetilde{\beta}_g([D\psi: D\varphi]_t)u_g =u_g \alpha_g([D\psi: D\varphi]_t).$$
Recall that we have cocycle relations:
\begin{align*}
	&c(g,h) =  {u}_{g}\alpha_{g}({u}_{h}) {u}_{gh}^*, \quad \text{for all }g,h\in \Gamma; \\
	& [D\psi: D\varphi]_{t+s} = [D\psi: D\varphi]_t\sigma_t^\varphi([D\psi:D\varphi]_s), \quad \text{for all }t,s\in \R.
\end{align*}
We then compute that for any $(g,t),(h,s)\in G$,
\begin{align*}
	\omega_{(g,t)} \alpha^\varphi_{(g,t)} (\omega_{(h,s)})
	&= u_g \alpha_g([D\psi: D\varphi]_t)\alpha_g\circ\sigma^\varphi_{t} ([D\psi:D\varphi]_s\sigma_s^{\varphi}(u_h))\\
	&= u_g \alpha_g([D\psi:D\varphi]_{t+s}\sigma^\varphi_{t+s}(u_h))\\
	&= u_g \alpha_g(w_{(h,t+s)})\\
	&= u_g \alpha_g(u_h \alpha_h([D\psi:D\varphi]_{t+s}))\\
	&= c(g,h)u_{gh} \alpha_{gh}([D\psi:D\varphi]_{t+s})\\
	&= c^\psi((g,t),(h,s))\omega_{(gh,t+s)},
\end{align*}
and similarly $\Ad(\omega_{(g,t)})\circ \alpha_{(g,t)}^\varphi = \widetilde{\beta}_{(g,t)}^\psi$. 
\end{proof}

	Now we put $G_1 := \R \leq G$. Then since we already have $(\C p, \sigma^\psi)\preceq_B (\C,\sigma^\varphi)$ for all projections $p\in   B_\psi^{\widetilde{\beta}}=B^{\widetilde{\beta}^\psi}$, we can apply Proposition \ref{prop for cocycle superrigidity}. 
Thus there exist a separable Hilbert space $H$, a projection $f\in \B(H)$, a $\sigma$-strongly continuous map $v\colon G=\Gamma \times \R \to \mathcal{U}(f\B(H)f)$, a partial isometry $w\in B\ovt \B(H)$ such that, 
	\begin{itemize}
		\item $wv_g =  (\omega_g\otimes 1_H )(\alpha_g^\varphi\otimes \id_H) (w)$ \quad for all $g\in G$;
		\item $w^*w = f$ and $ww^* = 1\otimes e_{1,1}$, where $e_{1,1}\in \B(H)$ is a minimal projection;
		\item $(\Ad(v_g))_{g\in G}$ and $(v_gv_h v_{gh}^*)_{g,h\in G}$ define a cocycle action on $f\B(H)f$;
		\item $\widetilde{\beta}^\psi_g( wxw^*) = w   (\alpha^\varphi_g\otimes \Ad(v_g))(x) w^*$ \quad for all $x\in B\ovt f\B(H)f.$
	\end{itemize}
As in the proof of Proposition \ref{prop for cocycle superrigidity}, the first equation implies $v_{t+s}=v_tv_s$ for all $t,s \in \R$, hence  $(v_t)_{t\in \R}$ is a continuous homomorphism. By Stone's theorem, there is a unique analytic generator $h$ on $fH$, so that $[\Tr_H(h\, \cdot \, ),f\Tr_H f]_t = h^{it} = v_t$ for all $t\in \R$, where $\Tr_H$ is a fixed semifinite trace on $\B(H)$ (with $\Tr_H(e_{1,1})=1$). 
We then compute that for all $t\in \R$, with $\varphi^H:= \varphi \otimes \Tr_H$, $\psi^H:=\psi \otimes \Tr_H$ and $h=1_B\otimes h$, using Lemma \ref{connes cocycle lemma},
\begin{align*}
	&[Df\varphi^H(h\, \cdot \, ) f: D\psi^H \circ \Ad(w)]_t\\
	=\ & [Df\varphi^H(h\, \cdot \, )f: Df\varphi^Hf]_t [Df\varphi^Hf: D\psi^H\circ \Ad(w)]_t \\
	=\ & v_t [Df\varphi^H f: Df\psi^Hf]_t [Df\psi^Hf: D\psi^H\circ \Ad(w)]_t \\
	=\ & v_t ([D\varphi:D\psi]_t \otimes 1_H) (\sigma_t^\psi\otimes \id_H)(w^*) w \\
	=\ & v_t (\sigma_t^\varphi\otimes \id_H)(w^*) ([D\varphi:D\psi]_t \otimes 1_H) w\\
	=\ & w^*([D\psi,D\varphi]_t\otimes 1_H) ([D\varphi:D\psi]_t \otimes 1_H) w\\
	=\ & f.
\end{align*}
We get that $\varphi^H(h\, \cdot \, )  = \psi^H \circ \Ad(w)$. 
In particular, putting $\mu:=\Tr_H(h \, \cdot \, )$, 
	$$\Ad(w^*)\colon B = B \otimes \C e_{1,1} \to B\ovt f\B(H)f$$ 
satisfies $\psi=(\varphi\otimes \mu)\circ \Ad(w^*)$. 
Since $\Ad(w^*)$ gives a conjugacy between $\alpha^\varphi\otimes \Ad(u)$ and $\widetilde{\beta}^\psi$, by restriction, it gives a state preserving conjugacy between $\alpha\otimes \Ad(u)$ and $\widetilde{\beta}$.

	Finally we show that $\Lambda_0 $ is a finite group. Observe that $\Tr_H(h) = \psi(1) <\infty$, so $h$ is a compact operator on $fH$. It holds that 
	$$A_\psi\rtimes_\beta \Lambda_0= (A\rtimes_\beta \Lambda_0)_\psi \simeq (B\ovt f\B(H)f)_{\varphi\otimes \mu}.$$
Since $h$ is a compact operator, there exist finite rank projections $r_n$ on $fH$ which commutes with $h$ such that $r_n \to f$. Then since $\sigma^\varphi$ is weakly mixing, one has $r_n(B\ovt f\B(H)f)_{\varphi\otimes \mu}r_n = \C \otimes (r_n\B(H)r_n)_\mu$ for all $n$. In particular $(B\ovt f\B(H)f)_{\varphi \otimes \mu}$ is an atomic von Neumann algebra, so that $A_\psi\rtimes_\beta \Lambda_0$ as well. This implies that $\Lambda_0$ is a finite group (and $A_\psi$ is atomic).
\end{proof}

\subsection*{Rigidity of Bernoulli shifts for genuine actions}

	We continue to use the Bernoulli shift action $\Gamma \curvearrowright^\alpha \bigotimes_{\Gamma}(B_0,\varphi_0)=(B,\varphi)$ and $M=B\rtimes_\alpha \Gamma$, assuming that $B_0$ is amenable.  
We recall the following fact. 

\begin{Thm}[{\cite[Theorem A]{Ma16}}]\label{Bernoulli theorem}
	Let $p \in M$ be a projection, $A \subset pMp$ a finite von Neumann subalgebra with expectation. 
\begin{itemize}
	\item[$\rm(1)$] If $A \not\preceq_M L\Gamma$, then $A' \cap pMp$ has an amenable direct summand. 
	\item[$\rm(2)$] If $A$ has relative property (T) in $pMp$, then $A \preceq_M L\Gamma$.
\end{itemize}
\end{Thm}

\begin{proof}[Proof of Proposition \ref{thmE}]
	By assumption, there are isomorphisms $\Gamma \simeq \Lambda$, $A\simeq B$, and there is a cocycle $\omega\colon \Gamma \to \mathcal{U}(B)$ such that  $\beta=\alpha^\omega$. 

	Assume that $\Gamma$ has a normal subgroup $\Gamma_1 \leq \Gamma$ with relative property (T). Put $\Lambda_1\leq \Lambda$ as the image of $\Gamma_1$. For any projection $q\in L\Lambda_1' \cap B$, we apply Theorem \ref{Bernoulli theorem}(2) to $L\Lambda_1 q$ and get that $L\Lambda_1 q \preceq_M L\Gamma$.

	Assume that $\Gamma$ is a direct product $\Gamma=\Gamma_1\times \Gamma_2$ with $\Gamma_2$ non-amenable. We put $\Lambda_i\leq \Lambda$ as images of $\Gamma_i$ for $i=1,2$. 
For any projection $q\in L\Lambda_1' \cap B$, we apply Theorem \ref{Bernoulli theorem}(1) to $L\Lambda_1 q$. We get that $L\Lambda_1 q \preceq_M L\Gamma$.

	Thus in both cases, one has $L\Lambda_1 q \preceq_M L\Gamma$ for any projection $q\in L\Lambda_1' \cap B$. Fix such $q\in L\Lambda_1' \cap B$ and we claim that $(\C q , \beta|_{\Lambda_1})\preceq_{B}(\C, \alpha|_{\Gamma_1})$. 
Indeed, suppose by contradiction that there is $(g_i)_{i\in I}$ in $\Lambda_1$ such that 
	$$\varphi(\alpha_{g_i}(b^*)\omega_{g_i}^*q a ) \rightarrow 0, \quad \text{$\sigma$-strongly for all $a,b\in B$.}$$ 
Then for any $a,b\in B$ and $s,s'\in \Gamma$, we have 
\begin{align*}
	 E_{L\Gamma}(\lambda_s^{\alpha} b^* \Pi^\omega_{\alpha,\beta}(\lambda_{g_i^{-1}}^\beta)q a\lambda_{s'}^\alpha)
	= \lambda_s^\alpha E_{L\Gamma}(b^*  \lambda_{g_i^{-1}}^\alpha \omega_{g_i}^* q a)\lambda_{s'}^\alpha  	
	= \lambda_{sg_i^{-1}}^\alpha \varphi(\alpha_{g_i}(b^*  )\omega_{g_i}^*q  a)\lambda_{s'}^\alpha.
\end{align*}
The last term converges to 0, hence we get $L\Lambda_1 q \not\preceq_M L\Gamma$, a contradiction.

	Finally since $\Lambda_1 \leq \Lambda$ is normal, we can apply Proposition \ref{prop for cocycle superrigidity} and get a cocycle action $(\Ad(u_g))_{g\in \Gamma}$ on a factor $\B$. By construction, this cocycle action is a genuine action and we finish the proof.
\end{proof}

\section{Strong solidity of free product factors}\label{Strong solidity of free product factors}

	For amalgamated free product von Neumann algebras and their modular theory, we refer the reader to \cite{VDN92,Ue98}. Throughout this section we fix the following setting. 

	Let $I$ be a set,  $(M_i)_{i\in I}$ a family of $\sigma$-finite von Neumann algebras, $B \subset M_i$ a common unital von Neumann subalgebra with expectations $E_i$ for all $i \in I$. Denote by $M:=*_{B} (M_i,E_i)_{i\in I}$ the amalgamated free product von Neumann algebra, and by $E_B\colon M \to B$ the canonical conditional expectation. 
For any subset $\mathcal F \subset I$, we denote by $M_{\mathcal F}:=*_{B} (M_i,E_i)_{i\in \mathcal F}$, and $E_{\mathcal{F}}\colon M \to M_\mathcal{F}$ is the canonical conditional expectation. 

	To prove Theorem \ref{thmF}, we first prove the following special case. This is a variant of Ioana's theorem \cite[Theorem 1.6]{Io12} (see also \cite{Va13,HU15}), and the proof uses a theorem in \cite{BHV15}.

\begin{Lem}\label{semifinite AFP lemma}
	Let $I=\{1,2\}$. Assume that there is a semifinite trace $\Tr_B$ on $B$ such that $\Tr_B\circ E_i$ are tracial for all $i\in I$. 
Then the conclusion of Theorem \ref{thmF} holds for any $p\in M$ and $A\subset pMp$ as in the statement, provided that $\Tr_B\circ E_B(p)<\infty$. 
\end{Lem}
\begin{proof}
	Recall that for any semifinite von Neumann algerbas, relative injectivity and relative semidiscreteness are the same conditions (see \cite[Theorem A.6]{Is17}). 
To prove this lemma, we follow the argument in the paragraph just before \cite[Theorem A.4]{HU15}. In this argument, we can apply \cite[Theorem 3.11]{BHV15}, instead of \cite[Theorem 1.6]{PV11}. Then all other proofs work if we exchange the normalizer algebra with the stable normalizer algebra. Thus the conclusion of \cite[Theorem A.4]{HU15} holds for the stable normalizer von Neumann algebra and the lemma is proven.
\end{proof}

\begin{proof}[Proof of Theorem \ref{thmF}]
	Suppose that $A \not\preceq_M B$ and $s\mathcal{N}_{pMp}(A)'' \not\preceq_M M_i$ for $i=1,2$. We will prove that $P:=s\mathcal{N}_{pMp}(A)''$ is injective relative to $B$ in $M$. 

	Let $E_A$ and $E_P$ be faithful normal conditional expectations for $A$ and $P$ respectively, $N$ the hyperfinite type $\rm III_1$ factor, and $\omega$ a faithful normal state such that $N_\omega' \cap N =\C$. 
Observe that $A' \cap pMp \subset A$ implies that $E_A$ and $E_P$ are unique normal expectations, hence it holds that $E_A\circ E_P=E_A$. Using this uniqueness and using Theorem \ref{thmA}, there exist $\psi$ which is preserved by $E_A,E_P$, and $\varphi$ which is preserved by $E_B,E_{M_i}$ for $i=1,2$, such that 
\begin{align*}
	&\Pi_{\varphi\otimes \omega,\psi\otimes \omega}(C_{\psi\otimes \omega}(A\ovt N)) \not \preceq_{C_{\varphi\otimes \omega}(M\ovt N)}C_{\varphi\otimes \omega}(B\ovt N), \\
	&\Pi_{\varphi\otimes \omega,\psi\otimes \omega}(C_{\psi\otimes \omega}(P\ovt N)) \not \preceq_{C_{\varphi\otimes \omega}(M\ovt N)}C_{\varphi\otimes \omega}(M_i\ovt N), \quad \text{for }i=1,2.
\end{align*}
Observe that, since $A\ovt N$ is properly infinite, by \cite[Lemma 2.4]{FSW10}
	$$A\ovt N \subset P \ovt N \subset s\mathcal{N}_{pMp \ovt N}(A\ovt N)'' = \mathcal{N}_{pMp \ovt N}(A\ovt N)''.$$
In particular the inclusion $A\ovt N \subset P\ovt N$ is regular, and hence  by \cite[Lemma 4.1]{BHV15}, the inclusion $C_{\psi\otimes \omega}(A\ovt N)\subset C_{\psi\otimes \omega}(P\ovt N)$ is regular as well. 
For notation simplicity, we omit $\Pi_{\varphi\otimes \omega,\psi\otimes \omega}$ and write as $\mathcal{M}:=C_{\varphi\otimes \omega}(M\ovt N)$, $\mathcal{M}_{i}:=C_{\varphi\otimes \omega}(M_{i}\ovt N)$ for $i=1,2$, $\mathcal{B}:=C_{\varphi\otimes \omega}(B\ovt N)$, and $\mathcal{A}:=C_{\psi\otimes \omega}(A\ovt N)$. 
Let $\mathcal{E}_i\colon \mathcal{M}_i \to \mathcal{B}$ be the faithful normal conditional expectation such that $\mathcal{E}_i|_{M_i\ovt N}=E_{i}\otimes \id_N$ and $\mathcal{E}|_{L\R_{\varphi}}=\id_{L\R_\varphi}$ and note that $\mathcal{M}$ has an amalgamated free product structure 
	$$\mathcal{M}=(\mathcal{M}_1,\mathcal{E}_1)*_{\mathcal B}(\mathcal{M}_2,\mathcal{E}_2).$$
In this setting, our assumptions are translated to that, 
$\mathcal A \not \preceq_{\mathcal M} \mathcal{B}$, 
$\mathcal{N}_{p\mathcal M p}(\mathcal{A})'' \not \preceq_{\mathcal M} \mathcal{M}_i$ for all $i=1,2$, and $\mathcal {A}$ is injective relative to $\mathcal B$ in $\mathcal M$ (use \cite[Corollary 3.6 and Theorem 3.2]{Is17}). 
Fix any projection $r\in L_{\psi\otimes \omega}\R$ such that $\Tr_{\psi\otimes \omega}(r)<\infty$, and observe that one has $r\mathcal A r \not \preceq_{\mathcal M} \mathcal{B}$ and 
$r\mathcal{N}_{p\mathcal M p}(\mathcal{A})''r \not \preceq_{\mathcal M} \mathcal{M}_i$ for all $i=1,2$. 
Using the inclusion $r\mathcal{N}_{p\mathcal M p}(\mathcal A)''r \subset s\mathcal{N}_{pr\mathcal M pr}(r\mathcal A r)''$ (e.g.\ \cite[Proposition 2.10]{FSW10}), by applying Lemma \ref{semifinite AFP lemma} to $r\mathcal{A}r \subset rp\mathcal{M}rp$, we get that $r\mathcal{N}_{p\mathcal M p}(\mathcal A)''r$ is injective relative to $\mathcal B$. Since $r$ is arbitrary, by \cite[Lemma 3.3(v)]{HI17}, we conclude that $\mathcal{N}_{p\mathcal M p}(\mathcal A)''$ is injective relative to $\mathcal B$ in $\mathcal M$. Since $\mathcal{N}_{p\mathcal M p}(\mathcal A)''$ contains $C_{\psi\otimes \omega}(P\ovt N)$ with expectation, by \cite[Theorem 3.2]{Is17}, it holds that $P\ovt N$ is injective relative to $B\ovt N$ in $M\ovt N$. Finally it is easy to see that $P$ is injective relative to $B$ in $M$. This is the conclusion.
\end{proof}

\begin{proof}[Proof of Corollary \ref{corG}]
	If $M$ is stably strongly solid, then since all $M_i$'s are von Neumann subalgebras with expectation, all $M_i$'s are stably strongly solid. We have to show the converse. 

	Let $p\in M$ be a projection and $A \subset pMp$ a diffuse amenable von Neumann subalgebra with expectation. We have to show that $P:=s\mathcal{N}_{pMp}(A)''$ is amenable. 
Since $pMp$ is solid by \cite[Theorem 6.1]{HU15}, $A'\cap pMp$ is amenable. Then as in the proof of \cite[Main theorem]{BHV15}, up to exchanging $A\vee (A'\cap pMp)$ by $A$, we may assume that $A'\cap pMp \subset A$. 
Let $z\in P$ be the unique projection such that $P(p-z)$ is amenable and $Pz$ has no amenable direct summand. We will deduce a contradiction by assuming that $z\neq 0$. In this case, using $Pz \subset s\mathcal{N}_{zMz}(Az)''$, up to exchanging $z$ by $p$, we may assume that $P$ has no amenable direct summand. 
	Define $M^\infty := M\ovt \B(\ell^2)$, $M^\infty_i := M_i\ovt \B(\ell^2)$, $A^\infty := A\ovt \B(\ell^2)$, and $E^\infty_i:=E_i \otimes \id_{\B(\ell^2)}$, and observe that $M^\infty = *_{\B(\ell^2)}(M_i^\infty, E_i^\infty)_{i\in I}$ and $s\mathcal{N}_{pM^\infty p}(A^\infty)'' = \mathcal{N}_{pM^\infty p}(A^\infty)''$ (since $A^\infty$ is properly infinite). 
Since $A^\infty$ is diffuse, we have $A^\infty \not \preceq_{M^\infty}  \B(\ell^2)$. 

	Suppose first that $I=\{1,2\}$. We can apply Theorem \ref{thmF} to $A^\infty \subset pM^\infty p$, and get that (ii) $\mathcal{N}_{pM^\infty p}(A^\infty)''\preceq_{M^\infty}  M_i^\infty$ for some $i\in \{1,2\}$ or (iii) $\mathcal{N}_{pM^\infty p}(A^\infty)''$ is amenable. 
If (iii) holds, then since $P\ovt \B(\ell^2) \subset \mathcal{N}_{pM^\infty p}(A^\infty)''$ is with expectation, we get that $P$ is amenable, a contradiction. 
Hence one has the condition (ii). Fix $i$ such that $\mathcal{N}_{pM^\infty p}(A^\infty)''\preceq_{M^\infty}  M_i^\infty$, and take $(H,f,\pi,w)$ witnessing this condition. Observe that  $\pi(A^\infty) \subset f(M_i^\infty \otimes \M_n)f$ is a diffuse amenable von Neumann subalgebra with expectation and that $\pi(P\ovt \B(\ell^2)) \subset \mathcal{N}_{f(M_i^\infty\otimes \M_n)f}(\pi(A^\infty))''$ is with expectation. 
Since $M_i$ is assumed to be stably strongly solid, 
$M_i^\infty \otimes \M_n$ is strongly solid by \cite[Corollary 5.2]{BHV15}. We thus get that $\pi(P\ovt \B(\ell^2))$ is amenable. Since $\pi$ is a normal $\ast$-homomorphism, $P$ has an amenable direct summand, a contradiction. 
We have thus proved this theorem in the case $I=\{1,2\}$.

	Now we prove the general case. Let $I$ be a general set and we put $M_{\mathcal{F}}:=*_{i\in \mathcal F}(M_i,\varphi_i)$ for any subset $\mathcal F \subset I$. 
We fix any finite subset $\mathcal {F}\subset I$ and observe that $M_{\mathcal F}$ is stably strongly solid by the result in the last paragraph. 
we apply the same argument as in the case $I=\{1,2\}$ to $A\subset pMp$ using the decomposition $M = M_{\mathcal F} * M_{\mathcal F^c}$. Then since $M_{\mathcal F}$ is stably strongly solid, the only possible condition is that $\mathcal{N}_{pM^\infty p}(A^\infty)''\preceq_{M^\infty}  M_{\mathcal F^c}^\infty$. 
By assuming that this condition holds for all finite subsets $\mathcal F \subset I$, we will deduce a contradiction.

	Since $P \ovt \B(\ell^2) \subset \mathcal{N}_{pM^\infty p}(A^\infty)''$, using \cite[Lemma 4.8]{HI15}, we indeed have that $P\ovt \B(\ell^2)\preceq_{M^\infty}  M_{\mathcal F^c}^\infty$ for all finite subsets $\mathcal F \subset I$. 
Then as in the proof of Theorem \ref{thmF}, by applying Theorem \ref{thmA} (and using $N\simeq N\ovt \B(\ell^2)$), one has $\mathcal{P} \preceq_{\mathcal {M}} \mathcal{M}_{\mathcal{F}^c}$ for all finite subsets $\mathcal F \subset I$, where we used similar notations to ones in the proof of Theorem \ref{thmF}, such as $\mathcal{P}:=C_{\psi\otimes \omega}(P\ovt N)$, $\mathcal{M}_{\mathcal F^c}:=C_{\varphi\otimes \omega}(M_{\mathcal F^c}\ovt N)$ for appropriate $E_P,\psi,\varphi$.

	Fix any projection $r\in L_{\psi\otimes \omega}\R$ such that $\Tr_{\psi\otimes \omega}(r)<\infty$. Fix any projection $z\in \mathcal P ' \cap p\mathcal Mp = (P' \cap pMp)_{\psi} = \mathcal{Z}(P)$ (e.g.\ Lemma \ref{III1 factor tensor lemma}). 
We will prove that $r\mathcal{P}rz \preceq_{\mathcal {M}} \mathcal{M}_{\mathcal{F}^c}$ for all finite subsets $\mathcal F \subset I$. Then using \cite[Proposition 4.2]{HU15}, this will imply the amenability of $r\mathcal Pr$ and hence the one of $\mathcal P$, a contradiction. 
To prove this condition, fix $\mathcal F$, $r$ and $z$. 
Observe that $Pz \subset s\mathcal{N}_{zMz}(Az)''$. Then since $Pz$ has no amenable direct summand, we can apply the same argument to $Az \subset Pz$ (as we applied to $A\subset P$), and get that $\mathcal{P}z\preceq_{\mathcal M}  \mathcal{M}_{\mathcal F^c}$. 
Since the central support  of $rz$ in $\mathcal P z$ is $z$, by \cite[Remark 4.2(3)]{HI15}, we get $r\mathcal{P}rz \preceq_{\mathcal {M}} \mathcal{M}_{\mathcal{F}^c}$. This is the desired condition.
\end{proof}

\small{

}
\end{document}